\documentclass[a4paper,11 pt]{article}
\usepackage[latin1]{inputenc}
\usepackage[english]{babel}
\usepackage{amsfonts}
\usepackage{amssymb}
\usepackage{amsmath}
\usepackage{amscd}
\usepackage{xypic}
\usepackage[all]{xy}
\usepackage{mathrsfs}
\usepackage{latexsym}
\usepackage{fullpage}
\usepackage[mathscr]{eucal}
\usepackage{float}
\usepackage{hyperref}
\usepackage{manyfoot}

\usepackage{amsthm}

\usepackage[active]{srcltx}

\newif\ifpdf
\ifx\pdfoutput\undefined
\pdffalse
\else
\ifnum\pdfoutput=0
\pdffalse
\else
\pdfoutput=1 \pdftrue
\fi
\fi

         \ifpdf
         \usepackage[pdftex]{graphicx}
         \else
         \usepackage{graphicx}
         \fi

\newtheorem{theo}{Theorem}[section]
\newtheorem{lem}[theo]{Lemma}
\newtheorem{prop}[theo]{Proposition}
\newtheorem{cor}[theo]{Corollary}
\newtheorem{defin}[theo]{Definition}
\newtheorem{assumption}[theo]{Assumption}

\theoremstyle{definition}

\newtheorem{exam}[theo]{Example}

\newtheorem{rem}[theo]{Remark}
\newtheorem*{theoA}{Theorem A}
\newtheorem*{theoB}{Theorem B}

\newcommand{\oo}{\mathcal{O}}

\newcommand{\mL}{\mathcal{L}}

\newcommand{\ZZ}{\mathbb{Z}}

\newcommand{\CC}{\mathbb{C}}

\newcommand{\EE}{\mathcal{E}} %Tschirnhausen non ha sezioni
\newcommand{\FF}{\mathcal{F}} %duale del Tschirnhausen e ha sezioni

\newcommand{\Ld}{L_{\delta}}
\newcommand{\wA}{\widehat{A}}

\newcommand{\wo}{\hat{o}}
\newcommand{\BlA}{\wA^{\sharp}}

\newcommand{\lr}{\longrightarrow}
\newcommand{\mP}{\mathbb{P}}

\title{On surfaces with $p_g=q=2$, $K^2=5$ and Albanese map \\ of
degree $3$}
\author{ Matteo Penegini, Francesco Polizzi}

\date{}

\begin{document}
%%%%%%%%%%%%%%%%%%%%%%%%%%%%%%%%%%%%%%%%%%%%%%%%%%%%%%%%%%%%%%%%%%%%%%%%%%%%%%%%%%%%%%%%%%%
%%%%%%%%%%%%%%%%%%%%%%%%%%%%%%%%%%%%%%%%%%%%%%%%%%%%%%%%%%%%%%%%%%%%%%%%%%%%%%%%%%%%%%%%%%%
\maketitle
%%%%%%%%%%%%%%%%%%%%%%%%%%%%%%%%%%%%%%%%%%%%%%%%%%%%%%%%%%%%%%%%%%%%%%%%%%%%%%%%%%%%%%%%%%%
%%%%%%%%%%%%%%%%%%%%%%%%%%%%%%%%%%%%%%%%%%%%%%%%%%%%%%%%%%%%%%%%%%%%%%%%%%%%%%%%%%%%%%%%%%%

\begin{center}
\emph{To Professors F. Catanese and C. Ciliberto on the occasion
of their 60th birthday}
\end{center}

%%%%%%%%%%%%%%%%%%%%%%%%%%%%%%%%%%%%%%%%%%%%%%%%%%%%%%%%%%%%%%%%%%%%%%%%%%%%%%%%%%%%%%%%%%%
\begin{abstract}
We construct a connected, irreducible component of the moduli
space of minimal surfaces of general type with $p_g=q=2$ and
$K^2=5$, which contains both examples given by Chen-Hacon and the
first author. This component is generically smooth of dimension
$4$, and all its points parametrize surfaces whose Albanese map is
a generically finite triple cover.
\end{abstract}

%\begin{center}
\Footnotetext{{}}{2000\textit{ Mathematics Subject Classification}.
14J29, 14J10, 14J60.}

\Footnotetext{{}}{\textit{ Keywords}. Surfaces of general type,
Albanese map.}

%\footnote{Version18/03/2011}
%\end{center}
%%%%%%%%%%%%%%%%%%%%%%%%%%%%%%%%%%%%%%%%%%%%%%%%%%%%%%%%%%%%%%%%%%%%%%%%%%%%%%%%%%%%%%%%%%%%
\section{Introduction} \label{sec.Introduction}

The classification of minimal, complex surfaces $S$ of general type
with small birational invariants is still far from being achieved;
nevertheless, the study of
such surfaces has produced in the last years a considerable amount
of results, see for instance the survey paper \cite{BCP06}. If we
assume $1=\chi(\mathcal{O}_S)=1-q+p_g$, that is $p_g=q$, and $S$
\emph{irregular}, that is $q>0$, then the inequalities of
Bogomolov-Miyaoka-Yau and Debarre imply $1 \leq p_g \leq 4$. If
$p_g=q=4$ then $S$ is a product of curves of genus 2, as shown by
Beauville in the appendix to \cite{D82}, while the case $p_g=q=3$
was understood through the work of several authors, see
\cite{CCML98}, \cite{HP02}, \cite{Pi02}. The classification becomes
more and more complicated as the value of $p_g$ decreases; indeed
already for $p_g=2$ one has only a partial understanding of
the situation.

Let us summarize what is known for surfaces with $p_g=q=2$ in
terms of $K^2_S$; in this case the inequalities mentioned above
yield $4 \leq K^2_S \leq 9$. The case $K_S^2=4$ was investigated
by the first author, who constructed three families of surfaces
which admit an isotrivial fibration, see \cite{Pe09}. Previously,
surfaces with these invariants were also studied
  by Ciliberto and Mendes-Lopes
  (in connection with the problem of birationality of the
 bicanonical map, see \cite{CML02}) and Manetti (in his work on
 the Severi conjecture, see \cite{Man03}). For $K_S^2=5$ there were so far
 only two examples, see \cite{CH06}
 and \cite{Pe09}. As the title suggests, the present work deals with
 this case. For $K_S^2=6$ there is only one example, see again \cite{Pe09}.
The study of the case $K_S^2=8$ was started by Zucconi in
\cite{Z03} and continued by the first author in \cite{Pe09}. They
produced a complete classification of surfaces with $p_g=q=2$ and
$K_S^2=8$
 which are isogenous to a product of curves; as a by-product, they
obtained the classification of all surfaces with these invariant
which are not of Albanese general type, i.e., such that the image
of the
Albanese map is a curve.  Finally, for $K_S^2=7$ and $K_S^2=9$ there are hitherto no
examples
known.

In this article we consider surfaces with
$p_g=q=2$ and $K^2_S=5$. Our work started when we noticed that the
surfaces constructed in \cite{CH06} and \cite{Pe09}
 have many features in common. More precisely, in both cases the
Albanese map $\alpha \colon S \to \textrm{Alb(S)}$ is a generically
finite triple cover, and the Albanese variety $\textrm{Alb}(S)$ is
an abelian surface with a polarization of type $(1, \, 2)$.
Moreover, $S$ contains a $(-3)$-curve, which is obviously contracted
by $\alpha$. We shall prove that Penegini's and Chen-Hacon's
examples actually belong to the same connected component of the
moduli space of surfaces of general type with $p_g=q=2$ and
$K_S^2=5$.

In order to formulate our results, let us introduce some
terminology. Let $S$ be a minimal surface of general type with
$p_g=q=2$ and $K_S^2=5$, such that its Albanese map $\alpha \colon
S \to \textrm{Alb}(S)$
 is a generically finite morphism of degree $3$. If one considers the Stein
factorization of $\alpha$, i.e.,
\begin{equation*}
S \stackrel{p}{\lr} \widehat{X} \stackrel{\hat{f}}{\lr}
\textrm{Alb}(S),
\end{equation*}
then the map $\hat{f} \colon \widehat{X} \to \textrm{Alb}(S)$ is a
flat triple cover, which can be studied by applying the techniques
developed in \cite{M85}. In particular, $\hat{f}$ is determined by
a rank $2$ vector bundle $\EE$ on $\textrm{Alb}(S)$, called the
\emph{Tschirnhausen bundle} of the cover, and by a global section
$\eta \in H^0(\textrm{Alb}(S), \, S^3\EE^{\vee} \otimes
\bigwedge^2 \EE)$. In the examples of \cite{CH06} and \cite{Pe09}
the surface $\widehat{X}$ is singular; nevertheless in both cases
the numerical invariants of $\EE$ are the same predicted by the
formulae of \cite{M85}, as if $\widehat{X}$ were smooth. This
leads us to introduce the definition of \emph{negligible
singularity} for a triple cover, which is similar to the
corresponding well-known definition for double covers, see
Definition \ref{def.neg.sing}. Then, inspired by the construction
in \cite{CH06}, we say that $S$ is a \emph{Chen-Hacon surface} if
there exists a polarization $\mathcal{L}$ of type $(1,\, 2)$ on
$\textrm{Pic}^0(S)=\widehat{\textrm{Alb}(S)}$ such that
$\EE^{\vee}$ is the Fourier-Mukai transform of the line bundle
$\mathcal{L}^{-1}$, see Definition \ref{def.gch}.

Our first main result is the following characterization of
 Chen-Hacon surfaces, see Proposition \ref{prop.ch} and
Theorem \ref{teo.ch}.

\begin{theoA}
Let $S$ be a minimal surface of general type with $p_g=q=2$ and
$K_S^2=5$ such that the Albanese map $\alpha \colon S \to
\textrm{Alb}(S)$ is a generically finite morphism of degree $3$. Let
\begin{equation*}
S \stackrel{p}{\lr} \widehat{X} \stackrel{\hat{f}}{\lr}
\textrm{Alb}(S)
\end{equation*}
be the Stein factorization of $\alpha$. Then $S$ is a Chen-Hacon
surface if and only if $\widehat{X}$ has only negligible
singularities.
\end{theoA}

Moreover, we can completely describe all the possibilities for the
singular locus of $\widehat{X}$, see Proposition
\ref{prop.quadruple}. It follows that $\widehat{X}$ is never smooth,
since it always contains a cyclic quotient singularity of type
$\frac{1}{3}(1, \, 1)$. Therefore $S$ always contains a
$(-3)$-curve, which turns out to be the fixed part of the canonical
system $|K_S|$, see Proposition \ref{prop.canonical}.

Now let $\mathcal{M}$ be the moduli space of surfaces with
$p_g=q=2$ and let $\mathcal{M}^{CH} \subset \mathcal{M}$ be the
subset whose points parametrize (isomorphism classes of)
 Chen-Hacon surfaces. Our second main result is the
following, see Theorem \ref{teo.moduli}.

\begin{theoB}
$\mathcal{M}^{CH}$ is an irreducible, connected, generically
smooth component of $\mathcal{M}$ of dimension $4$.
\end{theoB}

Since  Chen and Hacon constructed in \cite{CH06} only the
\emph{general} surface in $\mathcal{M}^{CH}$, we need considerable
work in order to establish Theorem B. Our proof uses in an essential
way the fact that the degree of the Albanese map is a topological
invariant of $S$, see \cite{Ca91}. As a by-product, we obtain some
results of independent interest about the embedded deformations of
$S$ in the projective bundle $\mathbb{P}(\EE^{\vee})$, see
Proposition \ref{prop.H.unobstructed}.

%, and we are able to relate the
%deformations of $S$ to the deformations of its Albanese
%map, see Proposition \ref{prop.def.beta}.

We believe that the interest of our paper is twofold. First of
all, it provides the first construction of a connected component
of the moduli space of surfaces of general type with $p_g=q=2$,
$K_S^2=5$. Secondly, Theorem B shows that every small deformation
of a Chen-Hacon surface is still a Chen-Hacon surface; in
particular, no small deformation of $S$ makes the $(-3)$-curve
disappear. Moreover, since $\mathcal{M}^{CH}$ is generically
smooth, the same is true for the first-order deformations. By
contrast, Burns and Wahl proved in \cite{BW74} that first-order
deformations always smooth all the $(-2)$-curves, and Catanese
used this fact in \cite{Ca89} in order to produce examples of
surfaces of general type with everywhere non-reduced moduli
spaces. Theorem B demonstrates rather strikingly that the results
of Burns-Wahl and Catanese cannot be extended to the case of
$(-3)$-curves and, as far as we know, provides the first explicit
example of this situation.

%Last but not least, Chen-Hacon surfaces provide interesting examples
%of minimal surfaces of general type whose canonical system is
%composed with a rational pencil of curves of genus $3$, see
%Proposition \ref{prop.canonical}.

Although Theorems A and B shed some light on the structure of
surfaces with $p_g=q=2$ and $K_S^2=5$, many questions still remain
unanswered. For instance:
\begin{itemize}
\item Are there surfaces with these invariants whose Albanese map
has degree different from $3$?
\item Are there surfaces with these
invariants whose Albanese map has degree $3$, but which are not
Chen-Hacon surfaces? Because of Theorem A, this is the
same to ask whether $\widehat{X}$ may have non-negligible
singularities.
\end{itemize}
And, more generally:
\begin{itemize}
\item How many connected components of the moduli space of
surfaces with $p_g=q=2$ and $K_S^2=5$ are there?
\end{itemize}
In order to answer the last question, it would be desirable to find
an effective bound for the degree of $\alpha \colon S \to
\textrm{Alb}(S)$, but so far we
have not been able to do this.

Another problem that arises quite naturally and which is at
present unsolved is the following.
\begin{itemize}
\item What are the possible degenerations of Chen-Hacon
surfaces?
\end{itemize}
An answer to this question would be a major step toward a
compactification of $\mathcal{M}^{CH}$. In Proposition
\ref{reducible} we give a partial result, analyzing some
degenerations of the triple cover $\hat{f} \colon \widehat{X} \to
\textrm{Alb}(S)$ which provide reducible, non-normal surfaces.

Now let us describe how this paper is organized. In Section
\ref{sec.Preliminaries} we present some preliminaries, and we set up
notation and terminology. In particular we recall Miranda's theory
of triple covers, introducing the definition of negligible
singularity, and we discuss the geometry of $(1,2)$-polarized
abelian surfaces. For the reader's convenience, we recall the
relevant material from \cite{M85} and \cite{Ba87} without proofs,
thus making our exposition self-contained.

In Section \ref{sec.vec}, which is the technical core of the
paper, we describe all possibilities for the Tschirnhausen bundle
of the triple cover $\hat{f} \colon \widehat{X} \to \wA$. The
analysis is particularly subtle in the case where the
$(1,\,2)$-polarization is of product type; eventually, we are able
to rule out this case, showing that it gives rise to a surface
$\widehat{X}$ which is not of general type (see Corollaries
\ref{cor.nonsimple} and \ref{cor.simple}).

In Section \ref{sec.degree.3} we briefly explain the two examples
from \cite{CH06} and \cite{Pe09}, which motivate our definition of
Chen-Hacon surfaces. The properties of such surfaces are then
investigated in detail in Section \ref{sec.CH}.

Finally, in Section \ref{sec.main.thm} we prove Theorem A, whereas
Section \ref{sec.moduli} deals with the proof of Theorem B.

\bigskip

 \textbf{Acknowledgments.} Both authors are indebted with F. Catanese
 for suggesting the problem and for many enlighting conversations and
useful remarks.

They also thank C. Ciliberto, M. Manetti, M. Reid, S. Rollenske
and E. Sernesi for stimulating discussions.

M. Penegini was partially supported by the DFG Forschergruppe 790
 \emph{Classification of algebraic surfaces and compact complex
manifolds}.

F. Polizzi was partially supported by the World Class University
program through the National Research Foundation of Korea funded by
the Ministry of Education, Science and Technology
(R33-2008-000-10101-0), and by Progetto MIUR di Rilevante Interesse
Nazionale \emph{Geometria delle Variet$\grave{a}$ Algebriche e loro
Spazi di Moduli}. He thanks the Department of Mathematics of Sogang
University (Seoul, South Korea) and especially Yongnam Lee for the
invitation and the warm hospitality. He is also grateful to the
Mathematisches Institut-Universit\"at Bayreuth and to the Warwick
Mathematics Institute for inviting him in the period May-June 2010.

Both authors wish to thank the referee for many detailed comments and suggestions
that considerably improved the presentation of these results.
\bigskip

\textbf{Notation and conventions.} We work over the field $\mathbb{C}$
of complex numbers.

If $A$ is an abelian variety and $\wA:= \textrm{Pic}^0(A)$ its dual,
we denote by $o$ and $\hat{o}$ the zero point of $A$ and
$\widehat{A}$, respectively.

If $\mathcal{L}$ is a line bundle on $A$ we denote by
$\phi_{\mathcal{L}}$ the morphism $\phi_{\mL}: A \rightarrow \wA$
given by $x \mapsto t^*_x \mL \otimes \mL^{-1}$. If $c_1(\mL)$ is
non-degenerate then $\phi_{\mL}$ is an isogeny, and we denote by
$K(\mL)$ its kernel.

A coherent sheaf $\mathcal{F}$ on $A$ is called a \emph{IT-sheaf of
index i} if
\begin{equation*}
H^j(A, \, \mathcal{F} \otimes \mathcal{Q})=0 \quad \textrm{for all } \mathcal{Q} \in \textrm{Pic}^0(A)
\quad \textrm{and } j\neq i.
\end{equation*}
If $\mathcal{F}$ is an IT-sheaf of
index $i$ and $\mathcal{P}$ it the normalized Poincar\'e bundle on $A \times \wA$, the coherent sheaf
\begin{equation*}
\widehat{\mathcal{F}}:={R}^i\pi_{\wA \, *}(\mathcal{P} \otimes
\pi^*_{A}\mathcal{F})
\end{equation*}
is a vector bundle of rank $h^i(A, \, \mathcal{F})$, called the \emph{Fourier-Mukai transform} of $\mathcal{F}$.

By ``surface'' we mean a projective, non-singular surface $S$, and
for such a surface $\omega_S=\oo_S(K_S)$ denotes the canonical
class, $p_g(S)=h^0(S, \, \omega_S)$ is the \emph{geometric genus},
$q(S)=h^1(S, \, \omega_S)$ is the \emph{irregularity} and
$\chi(\mathcal{O}_S)=1-q(S)+p_g(S)$ is the \emph{Euler-Poincar\'e
characteristic}. If $q(S)>0$, we denote by $\alpha \colon S \to
\textrm{Alb}(S)$ the Albanese map of $S$.

If $|D|$ is any linear system of curves on a surface, its base
locus will be denoted by $\textrm{Bs}|D|$. If $D$ is any divisor,
$D_{\textrm{red}}$ stands for its support.

If $Z$ is a zero-dimensional scheme, we denote its length by
$\ell(Z)$.

If $X$ is any scheme, by ``first-order deformation" of $X$ we mean a
deformation over $\textrm{Spec}\,
\mathbb{C}[\epsilon]/(\epsilon^2)$, whereas by ``small deformation"
we mean a deformation over a disk $\mathcal{B}_r=\{ t \in \mathbb{C}
\, | \, |t| < r \}$.

%%%%%%%%%%%%%%%%%%%%%%%%%%%%%%%%%%%%%%%%%%%%%%%%%%%%%%%%%%%%%%%%%%%%%%%%%%%%%%%%%%%%%%%%%%%%
\section{Preliminaries} \label{sec.Preliminaries}

%%%%%%%%%%%%%%%%%%%%%%%%%%%%%%%%%%%%%%%%%%%%%%%%%%%%%%%%%%%%%%%%%%%%%%%%%%%%%%%%%%%%%%%%%%%%
\subsection{Triple covers of surfaces}
The theory of triple covers in algebraic geometry was developed by
R. Miranda in his paper \cite{M85}, whose main result is the
following.

\begin{theo} \emph{\cite[Theorem 1.1]{M85}} \label{teo.miranda}
A triple cover $f \colon X \to Y$ of an algebraic variety $Y$ is
determined by a rank $2$ vector bundle $\EE$ on $Y$ and by a global
section $\eta \in H^0(Y, \, S^3 \EE^{\vee} \otimes \bigwedge^2
\EE)$, and conversely.
\end{theo}

The vector bundle $\EE$ is called the \emph{Tschirnhausen bundle} of
the cover, and it satisfies
\begin{equation*}
f_{*}\oo_X = \oo_Y \oplus \EE.
\end{equation*}
 In the case of
smooth surfaces, one has the following formulae.

\begin{prop} \emph{\cite[Proposition 10.3]{M85}} \label{prop.invariants}
Let $f \colon S \rightarrow Y$ be
a triple cover of smooth surfaces with Tschirnhausen bundle $\EE$.
Then
\begin{itemize}
\item[$\boldsymbol{(i)}$] $h^i(S, \, \mathcal{O}_S)=h^i(Y, \,
\mathcal{O}_Y)+h^i(Y, \, \EE)$ for all $i\geq 0;$
\item[$\boldsymbol{(ii)}$]
$K^2_S=3K^2_Y-4c_1(\EE)K_Y+2c_1^2(\EE)-3c_2(\EE)$.
\end{itemize}
\end{prop}

Let $f \colon X \to Y$ be a triple cover, and let us denote by $D
\subset Y$ and by $R \subset X$ the branch locus and the
ramification locus of $f$, respectively. By \cite[Proposition
4.7]{M85}, $D$ is a divisor whose associated line bundle is
$\bigwedge^2 \EE^{\vee}$. If $Y$ is smooth, then $f$ is smooth over
$Y - D$, in other words all the singularities of $X$ come from the
singularities of the branch locus. More precisely, we have

\begin{prop} \emph{\cite[Proposition 5.4]{Pa89}} \label{prop.sing.ram}
Let $y \in \emph{Sing}(D)$. Then $X$ is singular over $y$ if and
only if one of the following conditions holds:
\begin{itemize}
\item[$\boldsymbol{(i)}$] $f$ in not totally ramified over $y;$
\item[$\boldsymbol{(ii)}$] $f$ is totally ramified over $y$ and
$\emph{mult}_y(D) \geq 3$.
\end{itemize}
\end{prop}

\begin{prop} \emph{\cite[Theorem 4.1]{Tan02}} \label{prop.can.ris}
Let $f \colon X \to Y$ be a triple cover of a smooth surface $Y$,
with
 $X$ normal. Then there are a finite number of blow-ups
$ \sigma \colon \widetilde{Y} \to Y$ of $Y$ and a commutative diagram
\begin{equation} \label{dia.can}
\begin{xy}
\xymatrix{%%
\widetilde{X}  \ar[d]_{\tilde{f}} \ar[rr]^{\tilde{\sigma}} & & X \ar[d]^{f} \\
\widetilde{Y}   \ar[rr]^{\sigma} & & Y,  \\
  }
\end{xy}
\end{equation}
where $\widetilde{X}$ is the normalization of $\widetilde{Y}
\times_{Y} X$, such that $\tilde{f}$ is a triple cover with smooth
branch locus. In particular, $\widetilde{X}$ is a resolution of
the singularities of $X$.
\end{prop}

We shall call $\widetilde{X}$ the \emph{canonical resolution} of the
singularities of $f \colon X \to Y$. In general, it does not
coincide with the \emph{minimal resolution} of the singularities of
$X$, which will be denoted instead by $S$.

\begin{defin} \label{def.neg.sing}
Let $f\colon X \rightarrow Y$ be a triple cover of a smooth
algebraic surface $Y$, with Tschirnhausen bundle $\EE$. We say that
$X$ has only \emph{negligible} $($or \emph{non essential}$)$
singularities if the invariants of the minimal resolution $S$ are
given by the formulae in Proposition \emph{\ref{prop.invariants}}.
\end{defin}

In other words, negligible singularities have no effect on the
computation of invariants. Let us give some examples.

\begin{exam} \label{ex.1}
Assume that the branch locus $D=D_{\textrm{red}}$ contains an
ordinary quadruple point $p$ over which $f$ is totally ramified. In
this case $\widetilde{Y}$ is the blow-up of $Y$ at $p$, and one sees
that the exceptional divisor is not in the branch locus of
$\tilde{f}$. We have $S=\widetilde{X}$, and the inverse image of the
exceptional divisor on $\widetilde{X}$ is a $(-3)$-curve. Therefore
$X$ has a singular point of type $\frac{1}{3}(1,1)$ over $p$, and by
straightforward computations (see \cite[Section 6]{Tan02}) one
checks that it is a negligible singularity.
\end{exam}

\begin{exam} \label{ex.2}
Assume that the branch locus $D=D_{\textrm{red}}$ contains an
ordinary double point $p$. A standard topological argument shows
that $X$ cannot be smooth over $p$, so Proposition
\ref{prop.sing.ram} implies that $p$ is not a point of total
ramification for $f$. Again, $\widetilde{Y}$ is the blow-up of $Y$
at $p$ and the exceptional divisor is not in the branch locus of
$\tilde{f}$. The inverse image of the exceptional divisor on
$\widetilde{X}$ consists of the disjoint union of a $(-1)$-curve and
a $(-2)$-curve; then the canonical resolution $\widetilde{X}$ does
not coincide with the minimal resolution $S$, which is obtained by
contracting the $(-1)$-curve. It follows that $X$ has both a smooth
point and a singular point of type $\frac{1}{2}(1,1)$ over $p$, and
as in the previous case one checks that this is a negligible
singularity for $X$.
\end{exam}

\begin{exam} \label{ex.3}
Assume that $f$ is totally ramified, that is $D = 2D_{\textrm{red}}$.
If $D_{\textrm{red}}$ contains an ordinary
double point, then over this point $X$ can have
either a singularity of type $\frac{1}{3}(1, \, 1)$ or
a singularity of type $\frac{1}{3}(1, \, 2)$.
It is no difficult to check that in both cases we have a
 negligible singularity.
\end{exam}

\begin{rem} \label{rem.negligible}
The definitions of canonical resolution and negligible singularity
for a triple cover are similar to the corresponding definitions
for double covers, that can be found for instance in \cite[Chapter
V]{BHPV03}. However, in contrast with the double cover
case, negligible singularities for triple covers are not
necessarily Rational Double Points, see for instance Example
\ref{ex.1}.
\end{rem}

%%%%%%%%%%%%%%%%%%%%%%%%%%%%%%%%%%%%%%%%%%%%%%%%%%%%%%%%%%%%%%%%%%%%%%%%%%%%%%%%%%%%%%%%%%%%

\subsection{Abelian surfaces with $(1,2)$ polarization} \label{sec.ab.surf}

Let $A$ be an abelian surface and $L$ an ample divisor in $A$ with
$L^2=4$. Then $L$ defines a polarization $\mL:=\oo_A(L)$ of type
$(1, \, 2)$, in particular $h^0(A, \, \mL)=2$ so the linear system
$|L|$ is a pencil. Such surfaces have been investigated by several
authors, see for instance \cite{Ba87}, \cite{HvM89}, \cite[Chapter
10]{BL04} and \cite{BPS09}. Here we just recall the results we need.

\begin{prop} \emph{\cite[p. 46]{Ba87}} \label{prop.barth.class}
Let $(A, \, \mL)$ be a $(1, \, 2)$-polarized abelian surface, with
$\mL=\oo_A(L)$, and let $C \in |L|$. Then we are in one of the
following cases:
\begin{itemize}
\item[$\boldsymbol{(a)}$] $C$ is a smooth, connected curve of
genus $3;$ \item[$\boldsymbol{(b)}$] $C$ is an irreducible curve
of geometric genus $2$, with an ordinary double point$;$
\item[$\boldsymbol{(c)}$] $C=E+F$, where $E$ and $F$ are elliptic
curves and $EF=2;$ \item[$\boldsymbol{(d)}$] $C=E+F_1+F_2$, with
$E$, $F_1$, $F_2$ elliptic curves such that $EF_1=1$, $EF_2=1$,
$F_1F_2=0$.
\end{itemize}
Moreover, in case $(c)$ the surface $A$ is isogenous to a product
of two elliptic curves, and the polarization of $A$ is the
pull-back of the principal product polarization, whereas in case
$(d)$ the surface $A$ itself is a product $E \times F$ and
$\mL=\mathcal{O}_A(E+2F)$.
\end{prop}

Let us denote by $\mathcal{W}(1,2)$ the moduli space of
$(1,2)$-polarized abelian surfaces; then there exists a Zariski
dense open set $\mathcal{U} \subset \mathcal{W}(1,2)$ such that,
given any $(A, \, \mL) \in \mathcal{U}$, all divisors in $|L|$ are
irreducible, i.e., of type $(a)$ or $(b)$, see \cite[Section
3]{BPS09}.

\begin{defin}
If $(A, \, \mL) \in \mathcal{U}$, we say that $\mL$ is a
\emph{general} $(1, \,2)$-polarization. If $|L|$ contains some
divisor of type $(c)$, we say that $\mL$ is a \emph{special} $(1,
\,2)$-polarization. Finally, if the divisors in $|L|$ are of type
$(d)$, we say that $\mL$ is a \emph{product} $(1,
\,2)$-polarization.
\end{defin}

If $\mL$ is not a product polarization, then $|L|$ has four distinct
base points $\{e_0,\, e_1, \, e_2, \, e_3\}$, which form an orbit
for the action of $K(\mL)\cong (\ZZ/2\ZZ)^2$ on $A$.
Moreover all curves in
$|L|$ are smooth at each of these base points, see \cite[Section
1]{Ba87}. There is
also a natural action of $K(\mL)$ on $|L|$, given by translation.

Let us denote by $(-1)_A$ the involution $x \to -x$ on $A$. Then we
say that a divisor $C$ on $A$ is \emph{symmetric} if $(-1)^*_A C =
C$. Analogously, we say that a vector bundle $\mathcal{F}$ on $A$ is
symmetric if $(-1)^*_A \FF = \FF$.

Since $\mathcal{L}$ is ample, \cite[Section 4.6]{BL04}
implies that, up to translation, it satisfies the following

\begin{assumption} \label{ass.sym}
$\mathcal{L}$ is symmetric and the base locus of
$|L|$ coincides with $K(\mL)$.
\end{assumption}

In the sequel we will tacitly suppose that Assumption
\ref{ass.sym} is satisfied.

\begin{prop}\label{prop.semichar} The following holds:
\begin{enumerate}
\item[$\boldsymbol{(i)}$] for all sections $s \in H^0(A, \,
\mL)$ we have $(-1)^*_{A}s=s$. In particular, all divisors in $|L|$
are symmetric;
 \item[$\boldsymbol{(ii)}$] we may assume $e_0=o$
and that  $e_1, \, e_2, \, e_3$ are $2$-division points,
 satisfying $e_1+e_2=e_3$.
\end{enumerate}
\end{prop}
\begin{proof} The first part of the statement follows from \cite[Corollary 4.6.6]{BL04},
whereas the second part follows from Assumption \ref{ass.sym}.
\end{proof}

\begin{prop} \label{prop.no.base.locus}
Let $\mathcal{Q}=\oo_A(Q) \in \emph{Pic}^0(A)$ be a non-trivial,
degree $0$ line bundle. Then we have $o \notin \emph{Bs}|L + Q|$,
and moreover
\begin{equation*} %\label{eq.no.base.locus}
h^0(A, \, \mathcal{L} \otimes \mathcal{Q} \otimes \mathcal{I}_o)=1, \quad
h^1(A, \, \mathcal{L} \otimes \mathcal{Q} \otimes \mathcal{I}_o)=0, \quad
h^2(A, \, \mathcal{L} \otimes \mathcal{Q} \otimes \mathcal{I}_o)=0.
\end{equation*}
\end{prop}
\begin{proof}
Since $\mathcal{L}$ is ample, the line bundle $\mathcal{L} \otimes \mathcal{Q}$
is equal to $t_x^*\mathcal{L}$ for some $x \in A$. Then $o \in \textrm{Bs}|L + Q|$
 if and only if $x \in K(\mL)$, that is $\mathcal{L} \otimes \mathcal{Q} =\mathcal{L}$,
 which is impossible since $\mathcal{Q}$ is non-trivial. The rest of the proof follows
 by tensoring with $\mathcal{Q}$ the short exact sequence
\begin{equation*}
0 \to \mL \otimes \mathcal{I}_o \to \mL \to \mL \otimes \oo_o \to 0
\end{equation*}
and by taking cohomology.
\end{proof}

In the rest of this section we assume that $\mathcal{L}$ is not a product
polarization. \\
We denote by $e_4, \ldots ,e_{15}$ the twelve $2$-division points
of $A$ distinct from $e_0$, $e_1$, $e_2$, $e_3$. Some of the
following results are probably known to the experts; however,
since we have not been able to find a comprehensive reference, for
the reader's convenience we give all the proofs.

\begin{prop}\label{prop.lin.M.sing} The following holds.
\begin{itemize}
\item[$\boldsymbol{(a)}$] Assume that  $\mL$ is a general $(1, \,
2)$-polarization. Then $|L|$ contains exactly $12$ singular curves
$L_5, \ldots, L_{16}$. Every $L_i$ has an ordinary double point at
$e_i$, and the set $\{L_i \}_{ i=4, \ldots, 15}$ consists of three
orbits for the action of $K(\mL)$ on $|L|$.
\item[$\boldsymbol{(b)}$] Assume that $\mL$ is a special $(1,
\,2)$-polarization, and let $E+F \in |L|$ be a reducible divisor.
Then the $K(\mL)$-orbit of $E+F$ consists of two curves $E+F$,
$E'+F'$ which intersect as in Figure
\emph{\ref{figure-reducible}}. Referring to this figure, the set
$\{p, \, q, \, r, \, s\}$ is contained in $\{e_4, \ldots,
e_{15}\}$, and it is an orbit for the action of $K(\mL)$ on $A$.
\end{itemize}
\end{prop}
\begin{figure}[H]
\begin{center}
\includegraphics*[totalheight=7 cm]{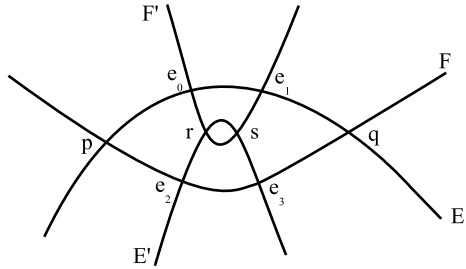}
\end{center}
\caption{The reducible curves $E+F$ and $E'+F'$ in the linear system
$|L|$} \label{figure-reducible}
\end{figure}

\begin{proof} $\boldsymbol{(a)}$ If a curve of $|L|$ contains any of the points
$e_4, \cdots, e_{15}$ then it must have a node there, see
\cite[Section 1.7]{Ba87} and \cite[Remark 11]{Yo97}. In order to
prove that there are no more singular curves, we blow-up the base
points of $|L|$ obtaining a genus $3$ fibration $ \tau \colon
\widetilde{A} \to \mathbb{P}^1$. By the Zeuthen-Segre formula, see
\cite[Lemma 6.4]{Be83}, we have
\begin{equation} \label{eq.ZS}
c_2(\widetilde{A})= e(\mathbb{P}^1)e(L)+ \sum (e(L_s)-e(L)),
\end{equation}
where the sum is taken on all the singular curves $L_s$ of $|L|$.
Since $e(L_s)=e(L)+1$ for a nodal curve, relation \eqref{eq.ZS}
implies that $|L|$ contains precisely $12$ singular elements. This
proves our first statement. The second statement is clear since
the twelve points $e_4, \ldots ,e_{15}$ consist of three orbits
for the action of $K(\mL)$ on $A$.

$\boldsymbol{(b)}$ Both curves $E$ and $F$ are fixed by the
involution $(-1)_{A}$, so they must both contain exactly four
$2$-division points. In particular the two intersection points of
$E$ and $F$ must be $2$-division points, say $E \cap F = \{p, \, q
\}$. Since we have
\begin{equation*}
t^*_{e_0}E=t^*_{e_{1}}E=E, \quad t^*_{e_0}F=t^*_{e_{1}}F=F,
\end{equation*}
it follows that the orbit of $E+F$ contains exactly two elements,
namely $E+F$ and $E'+F'$ where
\begin{equation*}
E':= t^*_{e_2}E=t^*_{e_{3}}E, \quad F':=t^*_{e_2}F=t^*_{e_{3}}F.
\end{equation*}
Setting $E' \cap F' = \{r, \, s \}$, it is straightforward to
check that the set of $2$-division points
 $\{p, \, q, \, r, \, s\}$
is an orbit for the action of $K(\mL)$ on $A$.
\end{proof}

\begin{rem} \label{Yoshihara}
In case $(b)$ of Proposition \ref{prop.lin.M.sing}, if one makes
the further assumption that $A$ is not isomorphic to the product
of two elliptic curves, it is not difficult to see that $E+F$ and
$E'+F'$ are the unique reducible curves in $|L|$, and that the
singular elements of $|L|$ distinct from $E+F$ and $E'+F'$ are
eight irreducible curves $L_i$ which have an ordinary double point
at the $2$-division points of $A$ distinct from $e_0$, $e_1$,
$e_2$, $e_3$, $p$, $q$, $r$, $s$. Moreover, these curves form two
orbits for the action of $K(\mL)$ on $|L|$.

There exist examples of abelian surfaces which are isomorphic to the
product of two elliptic curves and which admit also a special $(1,
\,2)$-polarization $\mathcal{L}$ besides the product polarization,
see \cite{Yo97}. For such surfaces, the linear system $|L|$ could
possibly contain more than two reducible curves (hence, less than
eight irreducible nodal curves).
\end{rem}

The other special elements of the pencil $|L|$ are smooth
hyperelliptic curves; let us compute their number.

\begin{prop} \label{prop.hyp}
The following holds.
\begin{itemize}
\item[$\boldsymbol{(a)}$]
Assume that $\mL$ is a general $(1, \,2)$-polarization. Then $|L|$
contains exactly six smooth hyperelliptic curves.
\item[$\boldsymbol{(b)}$] Assume that $\mL$ is a special
$(1, \,2)$-polarization. Then $|L|$ contains at most four smooth
hyperelliptic curves. More precisely, the number of such curves is
given by $6- \nu$, where $\nu$ is the number of reducible curves in $|L|$.
\end{itemize}
In any case, the set of hyperelliptic curves is union of orbits for the action of
$K(\mL)$ on $|L|$, and each of these orbits has cardinality $2$.
\end{prop}
\begin{proof}
$\boldsymbol{(a)}$ We borrow the following argument from
\cite[Proposition 3.3]{BPS09}. Let us consider again the blow-up
$\widetilde{A}$ of $A$ at the four base points of $|L|$ and the
induced genus $3$ fibration $\tau \colon \widetilde{A} \to
\mathbb{P}^1$. By \cite[Sections 3.2 and 3.3]{R} there is an
equality
\begin{equation} \label{eq.hyp}
K_{\widetilde{A}}^2=3 \chi(\oo_{\widetilde{A}})-10+ \deg \mathcal{T},
\end{equation}
where $\mathcal{T}$ is a torsion sheaf on $\mathbb{P}^1$ supported
over the points corresponding to the hyperelliptic fibres of
$\tau$. Since $\mL$ is a general polarization, we can have only
 smooth hyperelliptic fibres and the contribution of each of them
 to $\deg \mathcal{T}$, which
 is usually called the Horikawa number, is equal
 to $1$. So \eqref{eq.hyp} implies that $\tau$ has exactly
 six smooth hyperelliptic fibres.
On the other hand $K(\mL)$ acts on the set of hyperelliptic curves
of $|L|$, so have three orbits of cardinality $2$.

$\boldsymbol{(b)}$ The Horikawa number of a reducible
curve in $|L|$
is equal to $1$, see \cite{AK00}, so \eqref{eq.hyp} implies
that $|L|$ contains precisely $6-\nu$ smooth
hyperelliptic curves. In particular, by Remark \ref{Yoshihara}, $|L|$
contains exactly six hyperelliptic curves if $A$ is not isomorphic to
the product of two elliptic curves. Since the hyperelliptic curves
have non-trivial stabilizer for the action of $K(\mL)$ on $|L|$ when
$\mL$ is a general polarization (see part $(a)$), by a
limit argument we deduce that this is also true when $\mL$ is a
special polarization. It follows that the orbit of each
hyperelliptic curve consists again of exactly two curves.
\end{proof}

\begin{prop} \label{prop.hyp.stab}
Let $(A, \, \mL)$ be a $(1, \,2)$-polarized abelian surface and let
$C \in |L|$. Then the stabilizer of $C$ for the action of $K(\mL)$
on $|L|$ is non-trivial if and only if either $C$ is a smooth
hyperelliptic curve or $C$ is a reducible curve $($in the latter
case, $\mL$ is necessarily a special polarization$)$.
\end{prop}
\begin{proof}
The action of $K(\mL)$ on $|L| \cong \mathbb{P}^1$ induces a
$(\ZZ/2\ZZ)^2$-cover $\mathbb{P}^1 \to \mathbb{P}^1$, which is
 branched in three points by the Riemann-Hurwitz formula. This
implies that there are exactly six elements of $|L|$ having
non-trivial stabilizer. Our claim is now an immediate consequence
of Proposition \ref{prop.hyp} and Proposition
\ref{prop.lin.M.sing}, part $(b)$.
\end{proof}

Let us consider the line bundle $\mL^2=\oo_A(2L)$. It is a
polarization of type $(2, \, 4)$ on $A$, hence $h^0(A, \,
\mL^2)=8$. Moreover, since $\mL$  satisfies Assumption
\ref{ass.sym}, the same is true for $\mL^2$. Let $H^0(A, \,
\mL^2)^+$ and $H^0(A, \,
 \mL^2)^-$ be the subspaces of invariant and anti-invariant sections for
$(-1)_A$, respectively. One proves that
\begin{equation*}
\dim H^0(A, \, \mL^2)^+=6, \;\; \dim H^0(A, \,\mL^2)^-=2,
\end{equation*}
see \cite[Section 2]{Ba87}.

\begin{prop} \emph{\cite[Section 5]{Ba87}} \label{prop.anti.inv}
The pencil $\mathbb{P}H^0(A, \, \mL^2)^-$ of anti-invariant sections
has precisely $16$ distinct base points, namely $e_0, \, e_1,
\ldots, e_{15}$. Moreover all the corresponding divisors are smooth
at these base points.
\end{prop}

The $12$ points $e_4, \ldots e_{15}$ form three orbits
for the action of $K(\mL)$ on $A$; without loss of generality, we may
assume that these orbits are
\begin{equation*}
\{e_4, \, e_5, \, e_6, \, e_7 \}, \quad \{e_8, \, e_9, \, e_{10}, \,
e_{11} \}, \quad \{e_{12}, \, e_{13}, \, e_{14}, \, e_{15} \}.
\end{equation*}
Now let us take the $2$-torsion line bundles
$\mathcal{Q}_i:=\oo_A(Q_i)$, $i=1, \, 2, \,3$ such that
\begin{equation} \label{eq.2-tors}
t_{e_4}^* \mL = \mL \otimes \mathcal{Q}_1, \quad
t_{e_8}^* \mL   = \mL \otimes \mathcal{Q}_2, \quad
t_{e_{12}}^* \mL  = \mL \otimes \mathcal{Q}_3.
\end{equation}
Then
\begin{equation*}
\begin{split}
\textrm{Bs} \, |L + Q_1| & =\{e_4, \, e_5, \, e_6, \, e_7 \}, \\
\textrm{Bs} \, |L + Q_2| & =\{e_8, \, e_9, \, e_{10}, \, e_{11} \}, \\
\textrm{Bs} \, |L + Q_3| & =\{e_{12}, \, e_{13}, \, e_{14}, \,
e_{15} \}.
\end{split}
\end{equation*}
Moreover, for all $i=1, \,2, \, 3$,
\begin{equation} \label{eq.Io-Io2}
h^0(\wA, \, \mL \otimes \mathcal{Q} \otimes \mathcal{I}_o)=
h^0(\wA, \, \mL \otimes \mathcal{Q} \otimes \mathcal{I}_o^2)=1.
\end{equation}
Let us call $N_i$, $i=1, \,2,\,3$, the unique curve in the pencil
$|L + Q_i|$ containing $o$ (and having a node there, see \eqref{eq.Io-Io2}).
If $\mL$ is a general
$(1, \,2)$-polarization then the $N_i$ are all irreducible, in
particular they are smooth outside $o$.

\begin{defin} \label{def.D}
We denote by $\mathfrak{D}$ the linear system
$\mathbb{P}H^0(A, \, \mL^2 \otimes \mathcal{I}_o^4)$.
Geometrically speaking, $\mathfrak{D}$ consists of the curves in $|2L|$ having a point
of multiplicity at least $4$ at $o$.
\end{defin}

\begin{prop} \label{prop.pencil}
The linear system $\mathfrak{D} \subset |2L|$ is a pencil whose
general element is irreducible, with an ordinary quadruple point at
$o$ and no other singularities.
\end{prop}
\begin{proof}
Since the sections corresponding to the three curves $2N_i$
obviously belongs to $H^0(A, \, \mathcal{L}^2 \otimes
\mathcal{I}^4_o)$, by Bertini theorem it follows that the general
element of $\mathfrak{D}$ is irreducible, and smooth outside $o$. On
the other hand, $(2L)^2=16$, so the singularity at $o$ is actually
an ordinary quadruple point. Blowing up this point, the strict
transform of the general curve in $\mathfrak{D}$ has
self-intersection $0$, so $\mathfrak{D}$ is a pencil.
\end{proof}

The following classification of the curves in $\mathfrak{D}$ will
be needed in the proof of Theorem \ref{teo.moduli}.

\begin{prop} \label{prop.pencil.I4}
Let $(A, \, \mL)$ be a $(1, \,2)$-polarized abelian surface, and let
$C \in \mathfrak{D}$. Then we are in one of the following cases:
\begin{itemize}
\item[$\boldsymbol{(a)}$] $C$ is an irreducible curve of geometric
genus $3$, with an ordinary quadruple point;
\item[$\boldsymbol{(b)}$] $C$ is an irreducible curve of geometric
genus $2$, with an ordinary quadruple point and an ordinary double
point;
\item[$\boldsymbol{(c)}$] $C=2C'$, where $C'$ is an
irreducible curve of geometric genus $2$ with an ordinary double
point;
\item[$\boldsymbol{(d)}$] $\mathcal{L}$ is a special
$(1,\,2)$-polarization and $C=2C'$, where $C'$ is the union of two
elliptic curves intersecting in two points.
\end{itemize}
\end{prop}
\begin{proof}
By Proposition \ref{prop.pencil} the general element of
$\mathfrak{D}$ is as in case $(a)$. Now assume first that $\mL$ is
a general polarization. Then $\mathfrak{D}$ contains the following
distinguished elements:
\begin{itemize}
\item three reduced, irreducible curves $B_1$, $B_2$, $B_3$ such
that $B_i$ has an ordinary quadruple point at $o$, an ordinary
double point at $e_i$ and no other singularities (see
\cite[Corollary 4.7.6]{BL04}). These curves are as in case $(b)$;
\item three non-reduced elements, namely $2N_1$, $2N_2$, $2N_3$. These
 curves are as in case $(c)$.
\end{itemize}
Moreover, all the other elements of $\mathfrak{D}$ are smooth outside $o$;
one can see this by blowing-up $o$ and applying Zeuthen-Segre formula as in
the proof of Proposition \ref{prop.lin.M.sing}.

Finally, assume that $\mL$ is a special polarization. Then there
is just one more possibility, namely $C=2C'$, where $C'$ is the
translate of a reducible curve $E+F \in |L|$ by a suitable
$2$-division point. This yields case $(d)$.
\end{proof}

\begin{prop} \label{prop.2M.sym}
Every $s \in H^0(A, \, \mL^2 \otimes \mathcal{I}^4_o)$ satisfies $(-1)_A^* s
=s$.
\end{prop}
\begin{proof}
Let $v_1 \in  H^0(A, \, \mL \otimes \mathcal{Q}_1)$, $v_2 \in  H^0(A, \,
\mL \otimes \mathcal{Q}_2)$ be sections corresponding to the curves $N_1$ and
$N_2$, respectively. Since $N_1$ and $N_2$ are invariant divisors,
it follows $(-1)_A^*\,v_1= \pm v_1$ and $(-1)_A^*
\,v_2= \pm v_2$. Therefore $(-1)_A^* \,v_1^2=v_1^2$ and $(-1)_A^*
\,v_2^2=v_2^2$. But $v_1^2$, $v_2^2$ form a basis for $H^0(A, \,
\mL^2 \otimes \mathcal{I}_0^4)$, so we are done.
\end{proof}

\begin{prop} \label{prop.2M}
We have
\begin{equation*}
h^0(A, \, \mL^2 \otimes \mathcal{I}_o^3)=
h^0(A, \, \mL^2 \otimes \mathcal{I}_o^4)=2.
\end{equation*}
Geometrically speaking, every curve in $|2L|$, having multiplicity at least
$3$ at $o$, actually has multiplicity $4$.
\end{prop}
\begin{proof}
By contradiction, suppose that $H^0(A, \, \mL^2 \otimes \mathcal{I}_o^4)$ is
strictly contained in $H^0(A, \, \mL^2 \otimes \mathcal{I}_o^3)$. Then there
exists $w \in H^0(A, \, \mL^2 \otimes \mathcal{I}_o^3)$, $w \notin H^0(A, \,
\mL^2 \otimes \mathcal{I}_o^4)$ such that the three sections $v_1^2, \,
v_2^2, \, w \in H^0(A, \, \mL^2)$ are linearly independent;
let us write
\begin{equation*}
w=w^+ + w^-, \quad \textrm{where} \; \; w^+ \in H^0(A, \,
\mL^2)^+ \; \; \textrm{and} \; \; w^- \in H^0(A, \,
\mL^2)^-.
\end{equation*}
Consider the sum
\begin{equation*}
s = v_1^2+ v_2^2 +  w=
v_1^2+ v_2^2 +  w^+ +  w^-  \in H^0(A, \,
\mL^2 \otimes \mathcal{I}_o^3);
\end{equation*}
then Proposition \ref{prop.2M.sym} implies
\begin{equation*}
(-1)^*_A \, s = v_1^2+ v_2^2 +  w^+ -  w^-.
\end{equation*}
On the other hand, $(-1)_A$ fixes the tangent cone at $o$
of the curve corresponding to $s$; hence $(-1)^*_A \, s$ also
vanishes of order at least $3$ in $o$, that is $ (-1)^*_A \, s \in
H^0(A, \, \mL^2 \otimes \mathcal{I}_o^3)$.  This implies
\begin{equation*}
w^+, \; w^- \in H^0(A, \, \mL^2 \otimes \mathcal{I}_o^3).
\end{equation*}
Since by assumption $w=w^+ + w^- \notin H^0(A, \, \mL^2 \otimes \mathcal{I}_o^4)$,
it follows that either $w^+ \notin H^0(A, \, \mL^2 \otimes \mathcal{I}_o^4)$
or $w^- \notin H^0(A, \, \mL^2 \otimes \mathcal{I}_o^4)$.
In the former case, the curve $W^+:=\textrm{div}(w^+)$ is a
$\emph{even}$ divisor (i.e., corresponding to an invariant
section) in $|2L|$ which has multiplicity exactly $3$ at $o$; but
this is impossible, since every even divisor in $|2L|$ has even
multiplicity at the $2$-division points of $A$, see
\cite[Corollary 4.7.6]{BL04}. In the latter case, the curve
$W^-:=\textrm{div}(w^-)$ is an \emph{odd} divisor (i.e.,
corresponding to an anti-invariant section) in $|2L|$ which has
multiplicity exactly $3$ at $o$; but this is again a
contradiction, since all the odd divisors in $|2L|$ are smooth at
the $2$-division points of $A$, see Proposition
\ref{prop.anti.inv}.
\end{proof}

\section{Computations on vector bundles} \label{sec.vec}

Let $(A, \, \mL)$ be a $(1,\,2)$-polarized abelian surface.
Throughout this section, $\FF$ will denote a rank $2$ vector bundle
on $A$ such that
\begin{equation} \label{eq.vec.F}
h^0(A, \,\FF)=1, \quad h^1(A, \, \FF)=0, \quad h^2(A, \, \FF)=0,
\quad \det \FF = \mathcal{L};
\end{equation}
note that \eqref{eq.vec.F} together with Hirzebruch-Riemann-Roch
implies $c_2(\FF)=1$. These results will be needed in Section
\ref{sec.main.thm}.

\begin{prop} \label{prop.decomp.unstable}
If $\mathcal{F}$ is the direct sum of two line bundles, then it cannot be strictly
$\mathcal{L}$-semistable.
\end{prop}
\begin{proof}
Set $\mathcal{F}=\oo_A(C_1) \oplus \oo_A(C_2)$, where $C_1$, $C_2$
are divisors in $A$, and suppose by contradiction that
$\mathcal{F}$ is $\mathcal{L}$-semistable. Since $L=C_1+C_2$, we
obtain
\begin{equation*}
C_1(C_1+C_2)=C_2(C_1+C_2)=2.
\end{equation*}
On the other hand $1=c_2(\FF)=C_1C_2$ and so $C_1^2=C_2^2=1$, which
is absurd.
\end{proof}

From now on, we assume that $\FF$ is
\emph{indecomposable}. We divide the rest of the section into three subsections according to the properties of $\mathcal{L}$ and $\mathcal{F}$.
\subsection{The case where $\mathcal{L}$ is not a product polarization}

\begin{prop} \label{prop.L.irred.F.simple}
If $\mathcal{L}$ is not a product polarization, then $\mathcal{F}$
 is isomorphic to the unique locally free extension
\begin{equation*}
0 \longrightarrow \mathcal{O}_A \longrightarrow \FF
\longrightarrow \mL \otimes \mathcal{I}_x \longrightarrow 0,
\end{equation*}
with $x \in K(\mathcal{L})$. Moreover, $\mathcal{F}$ is $\mathcal{H}$-stable for any ample line
bundle $\mathcal{H}$ on $A$.
\end{prop}
\begin{proof}
Since $h^0(A, \,\mathcal{F})=1$, there exists an injective morphism
of sheaves $\mathcal{O}_A \hookrightarrow \mathcal{F}$. By
\cite[Proposition 5 p. 33]{F98} we can find an effective divisor $C$
and a zero-dimensional subscheme $Z$ such that $\mathcal{F}$ fits
into a short exact sequence
\begin{equation}\label{suc.def.f}
0 \longrightarrow \mathcal{O}_A(C) \longrightarrow \FF
\longrightarrow \mathcal{I}_Z(L-C)  \longrightarrow 0.
\end{equation}
Then $h^0(A, \, \oo_A(C))=1$ and
\begin{equation} \label{eq.c2}
1=c_2(\FF)=C (L-C) + \ell(Z).
\end{equation}
Now there are three possibilities:
\begin{itemize}
\item[$(i)$] $C$ is an elliptic curve; \item[$(ii)$] $C$ is a
principal polarization; \item[$(iii)$] $C=0$.
\end{itemize}

In case $(i)$ we have $C^2=0$, then by \eqref{eq.c2} we obtain
$CL=1$ and $\ell(Z)=0$. Thus \cite[Lemma 10.4.6]{BL04} implies
that $\mL$ is a product polarization, contradiction.

In case $(ii)$, the Index Theorem yields $(CL)^2 \geq C^2L^2 =8$, so
using \eqref{eq.c2} we deduce $C L=3$, $\ell(Z)=0$. Setting
$\mathcal{C}:=\oo_A(C)$, sequence \eqref{suc.def.f} becomes
\begin{equation*}
0 \lr \mathcal{C} \lr \mathcal{F} \lr \mathcal{C}^{-1} \otimes
\mathcal{L} \lr 0.
\end{equation*}
Being $\mathcal{F}$ indecomposable by assumption, we have
\begin{equation} \label{eq.h1}
H^1(A, \, \mathcal{C}^2 \otimes \mathcal{L}^{-1})=
\textrm{Ext}^1(\mathcal{C}^{-1} \otimes \mathcal{L}, \,
\mathcal{C}) \neq 0.
\end{equation}
Moreover, since $(-2C + L) L=-2$, the divisor $-2C+L$ is not
effective, that is
\begin{equation} \label{eq.h0}
H^2(A, \, \mathcal{C}^2 \otimes \mathcal{L}^{-1})=H^0(A, \,
\mathcal{C}^{-2} \otimes \mathcal{L})=0.
\end{equation}
On the other hand, by Riemann-Roch we have
\begin{equation*}
\mathcal{\chi}(A, \, \mathcal{C}^2 \otimes \mathcal{L}^{-1})=
\frac{1}{2}(2C-L)^2=0,
\end{equation*}
so \eqref{eq.h1} and \eqref{eq.h0} yield $H^0(A, \, \mathcal{C}^2
\otimes \mathcal{L}^{-1}) \neq 0$. This implies that $2C - L$ is
effective, so by using \cite[Lemma 1.1]{Ba87} and the equality
$(2C-L)C=1$ one concludes that there exists an elliptic curve $E$
on $A$ such that $2C-L = E$. Thus \cite[Lemma 10.4.6]{BL04}
implies that $A$ is a product of elliptic curves and that $C$ is a
principal product polarization. In other words $A=E \times F$ and
$C$ is algebraically equivalent to $E+F$. But then $L$ is
algebraically equivalent to $E+2F$, contradicting the fact that
$\mathcal{L}$ is not a product polarization.

Therefore the only possibility is $(iii)$, namely $C=0$. It follows
that $Z$ consists of a single point $x \in A$ and, since $\FF$ is
locally free, $x$ is a base point of $|L|$, i.e., $x \in K(\mathcal{L})$.

Therefore \eqref{suc.def.f} becomes
\begin{equation} \label{suc.def.f.x}
0 \longrightarrow \mathcal{O}_A \longrightarrow \FF
\longrightarrow \mL \otimes \mathcal{I}_x \longrightarrow 0.
\end{equation}
Tensoring \eqref{suc.def.f.x} with $\FF^{\vee}$ and taking
cohomology, we obtain
\begin{equation*}
\begin{split}
1 & \leq h^0(A, \, \FF \otimes \FF^{\vee})=h^0(A, \, \FF^{\vee}
\otimes \mL \otimes \mathcal{I}_x) = h^0(A, \,
\FF^{\vee} \otimes \bigwedge^2 \FF \otimes \mathcal{I}_x) \\
&= h^0(A, \, \FF \otimes \mathcal{I}_x) \leq h^0(A, \, \FF) =1.
\end{split}
\end{equation*}
Therefore $H^0(A, \, \FF \otimes \FF^{\vee})= \mathbb{C} $, that
is $\FF$ is simple. Since $c^2_1(\FF)-4c_2(\FF)=0$, by
\cite[Proposition 5.1]{T72} and \cite[Proposition 2.1]{T73} it
follows that $\FF$ is $\mathcal{H}$-ample for any ample line
bundle $\mathcal{H}$ on $A$.

It remains to show that \eqref{suc.def.f.x} defines a unique
locally free extension. By applying the functor
$\textrm{Hom}(-,\mathcal{O}_A)$ to
\begin{equation} \label{eq.ext}
0 \longrightarrow \mL \otimes \mathcal{I}_x \longrightarrow \mL
\longrightarrow \mL \otimes \mathcal{O}_x \longrightarrow 0
\end{equation}
and using Serre duality, we get
\begin{equation*}
\begin{split}
0 \longrightarrow \textrm{Ext}^1(\mL \otimes \mathcal{I}_x, \,
\mathcal{O}_A) & \longrightarrow \textrm{Ext}^2(\mL \otimes
\mathcal{O}_x, \, \mathcal{O}_A) \cong H^0(A, \, \mL \otimes
\oo_x)^{\vee} \\ & \stackrel{\varphi}{\longrightarrow}
\textrm{Ext}^2(\mL, \, \mathcal{O}_A) \cong H^0(A, \, \mL)^{\vee}.
\end{split}
\end{equation*}
Being $x \in \textrm{Bs} \,|L|$, it follows that $\varphi$ is the
zero map (see \cite[Theorem 1.4]{Ca90}), so
\begin{equation} \label{eq.ext.1}
\textrm{Ext}^1(\mL \otimes \mathcal{I}_x, \, \mathcal{O}_A) =
\mathbb{C}.
\end{equation}
This completes the proof.
\end{proof}
\begin{rem}\label{rem.o}
Up to replacing $\mathcal{L}$ by $t^*_x\mathcal{L}$, which is still a symmetric $(1,2)$-polarization, we may assume $x=o$.
So $\mathcal{F}$ will be isomorphic to the unique locally free extension
 \begin{equation} \label{suc.def.f.triv}
0 \longrightarrow \mathcal{O}_A \longrightarrow \FF
\longrightarrow \mL \otimes \mathcal{I}_o \longrightarrow 0.
\end{equation}
\end{rem}
\begin{prop} \label{prop.F.WIT}
If $\mathcal{L}$ is not a product polarization, $\mathcal{F}$ is a
symmetric $\emph{IT}$-sheaf of index $0$.
\end{prop}
\begin{proof}
Since $\mL$ is a symmetric polarization, by applying $(-1)^*_A$ to
\eqref{suc.def.f.triv} we get
\begin{equation*}
0 \longrightarrow \mathcal{O}_A \longrightarrow (-1)^*_A \FF
\longrightarrow  \mL \otimes \mathcal{I}_o \longrightarrow 0.
\end{equation*}
But \eqref{eq.ext.1} implies that $\FF$ is the unique locally free
extension of $\mL \otimes \mathcal{I}_o$ by $\mathcal{O}_A$, so we
obtain $(-1)_A^* \FF = \FF$, that is $\FF$ is symmetric.

In order to prove that $\FF$ satisfies IT of index $0$, we must show that
\begin{equation} \label{eq.WIT}
\begin{split}
V^1(A, \, \FF) &:=\{\mathcal{Q} \in \textrm{Pic}^0(A)\; | \;
h^1(A, \, \FF \otimes
\mathcal{Q}) > 0\}=\emptyset \\
V^2(A, \, \FF)&:=\{\mathcal{Q} \in \textrm{Pic}^0(A)\; | \; h^2(A,
\, \FF \otimes \mathcal{Q}) > 0\}=\emptyset.
\end{split}
\end{equation}
First, notice that $\oo_A \notin V^1(A, \, \FF)$ and $\oo_A \notin
V^2(A, \, \FF)$,
 since $h^1(A, \, \FF) = h^2(A, \, \FF)= 0$. \\
Now take  $\mathcal{Q} \in \textrm{Pic}^0 (A)$ such that
$\mathcal{Q} \neq \oo_A$. Tensoring \eqref{suc.def.f.triv} with
$\mathcal{Q}$ and using
 Proposition \ref{prop.no.base.locus}, we obtain
\begin{equation*}
h^0(A, \, \mathcal{F} \otimes \mathcal{Q})=1, \quad h^1(A, \,
\mathcal{F} \otimes \mathcal{Q})=0, \quad h^2(A, \, \mathcal{F}
\otimes \mathcal{Q})=0.
\end{equation*}
Hence \eqref{eq.WIT} is satisfied, and the proof is complete.
\end{proof}

Since $\FF$ is simple and $\chi(A, \, \FF \otimes \FF^{\vee})=0$,
we have
\begin{equation} \label{eq.f.otim.fdul}
h^0(A, \, \FF \otimes \FF^{\vee})=1, \; \; h^1(A, \, \FF \otimes
\FF^{\vee})=2, \; \; h^2(A, \, \FF \otimes \FF^{\vee})=1.
\end{equation}
On the other hand, the Clebsch-Gordan formula for the tensor product
(\cite[p. 438]{At57}) gives an isomorphism
\begin{equation*}
\oo_A \oplus (S^2\FF \otimes \bigwedge^2 \FF^{\vee}) = \FF \otimes
\FF^{\vee},
\end{equation*}
so by using \eqref{eq.f.otim.fdul} we obtain
\begin{equation} \label{eq.S2=0}
h^0(A, \, S^2\FF \otimes \bigwedge^2 \FF^{\vee})=0, \; \; h^1(A, \,
S^2\FF \otimes \bigwedge^2 \FF^{\vee})=0, \; \; h^2(A, \, S^2\FF
\otimes \bigwedge^2 \FF^{\vee})=0.
\end{equation}

\begin{prop} \label{prop.S3}
If $\mL$ is not a product polarization, we have
\begin{equation} h^0(A, \, S^3\FF \otimes \bigwedge^2 \FF^{\vee} )=
h^0(A, \, \mathcal{L}^2 \otimes \mathcal{I}^3_o)=2.
\end{equation}
\end{prop}
\begin{proof} The Eagon-Northcott complex
applied to \eqref{suc.def.f.triv} yields
\begin{equation*}
0 \longrightarrow S^2\FF \otimes \bigwedge^2 \FF^{\vee}
\longrightarrow S^3\FF \otimes \bigwedge^2 \FF^{\vee}
\longrightarrow \mL^2 \otimes \mathcal{I}^3_o \longrightarrow 0,
\end{equation*}
so our assertion is an immediate consequence of  \eqref{eq.S2=0}
and Proposition \ref{prop.2M}.
\end{proof}

\subsection{The case where $\mathcal{L}$ is a product polarization and $\mathcal{F}$ is not simple}
Now let us assume that $\mL$ is a product $(1, \,2)$-polarization.
Then $A = E \times F$, where $E$ and $F$ are two elliptic curves,
whose zero elements are both denoted by $o$. Let $\pi_E \colon E
\times F \to E$ and $\pi _F \colon E \times F \to F$ be the natural
projections. For any $p \in F$ and $q \in E$, we will write $E_p$
and $F_q$ instead of $\pi_F^{-1}(p)$ and $\pi_E^{-1}(q)$.

Furthermore, up to translations we may assume
$\mathcal{L}=\oo_A(E_o+2F_o)$.

Following the terminology of  \cite{O71}, we say that $\FF$ is
\emph{of Schwarzenberger type} if it is indecomposable but not
simple.

\begin{prop} \label{prop.prod.non.ample}
Suppose that $\mL$ is a product $(1, \, 2)$-polarization. Then
$\mathcal{F}$ is of Schwarzenberger type if and only if it is a
non-trivial extension of the form
\begin{equation} \label{suc.nonsimple}
0 \lr \mathcal{C} \lr \FF \lr \mathcal{L} \otimes \mathcal{C}^{-1}
\lr 0,
\end{equation}
where $\mathcal{C}:=\oo_A(E_p+F_q)$, with $p \in F$ different from
$o$
and $q \in E$ a $2$-division point. \\
\end{prop}
\begin{proof}
If $\FF$ is a non-trivial extension of type \eqref{suc.nonsimple},
then \cite[Lemma p. 251]{O71} shows that $\FF$ is indecomposable but
$h^0(A, \, \FF \otimes \FF^{\vee})=2$, so $\FF$ is not simple. \\
Conversely, assume that  $\mathcal{F}$ is of Schwarzenberger type.
Being $\FF$ not simple, it is not $\mathcal{H}$-stable with
respect to any ample line bundle $\mathcal{H}$ on $A$. In
particular, $\FF$ is not $\mathcal{L}$-stable. An argument similar
to the one used in the proof of Proposition
\ref{prop.decomp.unstable} shows that $\FF$ is not strictly
$\mathcal{L}$-semistable, so it must be $\mathcal{L}$-unstable.
This implies that there exists a unique sub-line bundle
$\mathcal{C}:=\oo_A(C)$ of $\mathcal{F}$ with torsion-free
quotient such that
\begin{equation} \label{eq.CL}
2C L >  L^2 =4.
\end{equation}
Now let us write
\begin{equation} \label{eq.sch1}
0 \lr \oo_A(C) \lr \FF \lr \mathcal{I}_Z(L - C)\lr 0,
\end{equation}
where $Z \subset A$ is a zero-dimensional subscheme. Then by using
\eqref{eq.CL} we obtain
\begin{equation} \label{eq.c2(FF)}
1=c_2(\FF)=C(L-C)+\ell(Z) > 2-C^2 + \ell(Z),
\end{equation}
that is $C^2 > 1 + \ell(Z)$. On the other hand, since $h^0(A, \,
\mathcal{C})=1$, the only possibility is $\ell(Z)=0$ and $C^2=2$, in
particular $\mathcal{C}$ is a principal polarization. But
\eqref{eq.c2(FF)} also gives $3 = CL = C(E_o+2F_o)$, so $C$ is
numerically equivalent to $E_o + F_o$. Therefore we can write
$C=E_p+F_q$ for some $p \in F$, $q \in E$ and \eqref{eq.sch1}
becomes
\begin{equation} \label{eq.sch.2}
0 \to \oo_A(E_p + F_q) \to \FF \to \oo_A(E_o-E_p + 2F_o-F_q) \to
0.
\end{equation}
Since $h^0(A, \, \FF)=1$, we have $p \neq o$. On the other hand,
since $\FF$ is indecomposable, \eqref{eq.sch.2} must be non-split,
so
\begin{equation*}
H^1(A, \, \oo_A(2E_p-E_o+2F_q-2F_o)) \neq 0.
\end{equation*}
This implies that $2F_q$ is linearly equivalent to $2F_o$, that is
$q \in E$ is a $2$-division point.
\end{proof}

\begin{prop} \label{prop.no.product.Sch}
If $\mathcal{L}$ is a product polarization and $\FF$ is of
Schwarzenberger type, we have
\begin{equation*}
h^0(A, \, S^3 \FF \otimes \bigwedge^2 \FF^{\vee})= h^0(A, \, S^2 \FF
\otimes \bigwedge^2 \FF^{\vee} \otimes \mathcal{C})= h^0(A, \,  \FF
\otimes \bigwedge^2 \FF^{\vee} \otimes \mathcal{C}^2)=3.
\end{equation*}
\end{prop}
\begin{proof}
The Eagon-Northcott complex applied to \eqref{suc.nonsimple} gives
\begin{equation*}
\begin{split}
0 & \lr \FF \otimes \bigwedge^2 \FF^{\vee} \otimes \mathcal{C}^2 \lr
 S^2 \FF \otimes \bigwedge^2 \FF^{\vee} \otimes \mathcal{C} \lr
\mathcal{L} \otimes \mathcal{C}^{-1} \lr 0 \\
0 & \lr S^2 \FF \otimes \bigwedge^2 \FF^{\vee} \otimes \mathcal{C}
\lr S^3 \FF \otimes \bigwedge^2 \FF^{\vee} \lr \mathcal{L}^2 \otimes
\mathcal{C}^{-3} \lr 0.
\end{split}
\end{equation*}
On the other hand, we have
\begin{equation*}
\begin{split}
H^0(A, \, \mathcal{L} \otimes \mathcal{C}^{-1}) & =H^0(A, \,
\oo_A(E_o-E_p+F_q))=0, \\
H^0(A, \, \mathcal{L}^2 \otimes \mathcal{C}^{-3}) & = H^0(A, \,
\oo_A(2E_o-3E_p+F_q))=0.
\end{split}
\end{equation*}
Tensoring \eqref{suc.nonsimple} with $\bigwedge^2 \FF^{\vee} \otimes
\mathcal{C}^2$ we obtain $h^0(A, \, \FF \otimes \bigwedge^2
\FF^{\vee} \otimes \mathcal{C}^2)=3$, so the claim follows.
\end{proof}

\begin{cor} \label{cor.nonsimple}
If $\mathcal{L}$ is a product polarization and $\FF$ is of
Schwarzenberger type, then the natural product map
\begin{equation*}
H^0(A, \, \FF \otimes \bigwedge^2 \FF^{\vee} \otimes \mathcal{C}^2)
\otimes H^0(A, \, \FF \otimes \mathcal{C}^{-1})^{\otimes 2} \lr
H^0(A, \, S^3 \FF \otimes \bigwedge^2 \FF^{\vee})
\end{equation*}
is bijective. Therefore, if $f \colon X \to A$ is the triple cover
corresponding to a non-zero section $\eta \in H^0(A, \, S^3 \FF
\otimes \bigwedge^2 \FF^{\vee})$, the surface $X$ is reducible and
non-reduced.
\end{cor}
\begin{proof}
The first statement follows from Proposition
\ref{prop.no.product.Sch} and from $H^0(A, \, \FF \otimes
\mathcal{C}^{-1})= \mathbb{C}$. The second statement is an immediate
consequence of the first one, since $\eta$ can be written as
$\eta=\eta_1 \eta_2^2$, where $\eta_1 \in H^0(A, \, \FF \otimes
\bigwedge^2 \FF^{\vee} \otimes \mathcal{C}^2)$ and $\eta_2$ is a
generator of $H^0(A, \, \FF \otimes \mathcal{C}^{-1})$.
\end{proof}

\subsection{The case where $\mathcal{L}$ is a product polarization and $\mathcal{F}$ is simple}

\begin{prop} \label{prop.prod.ample}
Suppose that $\mL$ is a product $(1, \,2)$-polarization. Then the
following are equivalent:
\begin{itemize}
\item[$\boldsymbol{(i)}$] $\FF$ is simple;
\item[$\boldsymbol{(ii)}$] $\FF$ is $\mathcal{H}$-stable for any
ample line bundle $\mathcal{H}$ on $A;$
\item[$\boldsymbol{(iii)}$] there exists a $2$-division point
 $q \in E$ such that $\FF$ is isomorphic to the unique non-trivial
extension
\begin{equation*}
0 \lr \mathcal{O}_A(F_q) \lr \FF \lr \oo_A(E_o+F_q) \lr 0;
\end{equation*}
\item[$\boldsymbol{(iv)}$] there exists a $2$-division
point $q \in E$ such that $\FF(-F_q)=\pi_F^*\mathcal{G}$, where
$\mathcal{G}$ is the unique non-trivial extension
\begin{equation*}
0 \lr \oo_F \lr \mathcal{G} \lr \oo_F(o) \lr 0.
\end{equation*}
\end{itemize}
\end{prop}
\begin{proof}
$\boldsymbol{(i)} \Rightarrow \boldsymbol{(ii)}$ See
\cite[Proposition
5.1]{T72} and \cite[Proposition 2.1]{T73}. \\
$\boldsymbol{(ii)} \Rightarrow \boldsymbol{(iii)}$ If $\FF$ is
$\mathcal{H}$-stable, then it is simple. By \cite[Corollary p.
249]{O71}, there exists an abelian surface $B$, a degree $2$ isogeny
$\varphi \colon B \to A$ and a line bundle $\mathcal{N}:=\oo_B(N)$
on $B$ such that
\begin{equation} \label{eq:N}
\varphi_{*} \mathcal{N}= \FF.
\end{equation}
Let $\mathcal{Q} :=\oo_A(Q) \in \textrm{Pic}(A)$ be the $2$-torsion
line bundle defining the double cover $\varphi$; then
 the following equality holds in $\textrm{Pic}(A)$:
\begin{equation*}
\oo_{A}(E_o+2F_o)= c_1(\FF) = \oo_{A}(\varphi_*N + Q),
\end{equation*}
see \cite[Proposition 27 p. 47]{F98}. This implies
\begin{itemize}
\item $B=E \times \widetilde{F}$ and
\begin{equation*}
\varphi = id \times \tilde{\varphi} \colon E \times \widetilde{F}
\lr E \times F,
\end{equation*}
where $\tilde{\varphi} \colon \widetilde{F} \to F$ is a degree $2$
isogeny. Note that $Q = E_p - E_o$, where $p \in F$ is a
$2$-division point.
\item $N$ is a principal product polarization
 of the form $N=E_a + \widetilde{F}_q$, where $a \in \widetilde{F}$ is such
that $\tilde{\varphi}(a)=p$ and $q \in E$ is a $2$-division point.
\end{itemize}
Since $\widetilde{F}_q=\varphi^*F_q$, by using \eqref{eq:N} and
projection formula we obtain
\begin{equation*}
\varphi_* \oo_B(E_a)=\varphi_*(\mathcal{N}
(-\widetilde{F}_q))=\mathcal{F}(-F_q).
\end{equation*}
Thus $h^0(A, \,\mathcal{F}(-F_q))=1$, and so there exists an
injective morphism of sheaves $\oo_A(F_q) \hookrightarrow \FF$. Then
we can find an effective divisor $D$ on $A$ and a zero-dimensional
subscheme $Z \subset A$ such that $\FF$ fits into a short exact
sequence
\begin{equation*}
0 \lr \oo_A(F_q + D) \lr \FF \lr \mathcal{I}_Z(E_o+F_q-D) \lr 0.
\end{equation*}
Since $h^0(A, \, \oo_{A}(F_q + D))=H^0(A, \, \FF)=1$, either $D=0$
or $F_q + D$ is a principal product polarization. The latter
possibility cannot occur, otherwise $\FF$ would be of
Schwarzenberger type (Proposition \ref{prop.prod.non.ample}). Then
$D=0$ and $\ell(Z)=c_2(\FF)-(E_o+F_q)F_q=0$, so $Z$ is empty and
we are done. \\
$\boldsymbol{(iii)} \Rightarrow \boldsymbol{(iv)}$ We have $\mathcal{O}_A(E_o) = \pi^*_F\mathcal{O}_F(o)$. By \cite[Footnote ***, p.
257]{O71} the map
\begin{equation*}
\textrm{Ext}^1(\mathcal{O}_F(o), \mathcal{O}_F) \longrightarrow \textrm{Ext}^1(\mathcal{O}_A(E_o), \mathcal{O}_A)
\end{equation*}
is an isomorphism.  Since the unique not-trivial extension of $\mathcal{O}_F(o)$ with  $\mathcal{O}_F$ is $\mathcal{G}$, we get $\boldsymbol{(iv)}$.
\\
$\boldsymbol{(iv)} \Rightarrow \boldsymbol{(i)}$ Again, \cite[p.
257]{O71} gives
$\textrm{End}(\FF)=\textrm{End}(\mathcal{G})=\mathbb{C}$.
\end{proof}

\begin{prop} \label{prop.no.product.ample}
If $\mathcal{L}$ is a product polarization and $\FF$ is simple, we
have
\begin{equation*}
h^0(A, \, S^3 \FF \otimes \bigwedge^2 \FF^{\vee} \otimes
\oo_{A}(-F_q))=h^0(A, \, S^3 \FF \otimes \bigwedge^2 \FF^{\vee})=2.
\end{equation*}
\end{prop}
\begin{proof}
By Proposition \ref{prop.prod.ample}, we have $\FF(-F_q)= \pi_F^*
\mathcal{G}$, where $\mathcal{G}$ is the unique non-trivial
extension of $\oo_F(o)$ by $\oo_F$. Therefore
\begin{equation} \label{eq.H0.F.ample}
h^0(A, \, \FF(-F_q))=h^0(F, \, \pi_{F*}\FF(-F_q))=h^0(F, \,
\mathcal{G)}=1.
\end{equation}
By \cite[p. 438-439]{At57} we have
\begin{equation*}
S^2\mathcal{G}(-o)\oplus \mathcal{O}_A =\mathcal{G} \otimes \mathcal{G}^{\vee} = \mathcal{O}_A \oplus \mathcal{Q}_1  \oplus \mathcal{Q}_2 \oplus \mathcal{Q}_3
\end{equation*}
where the $\mathcal{Q}_i$ are the non trivial $2$-torsion line bundles on $A$. Since the decomposition of a vector bundle in indecomposable summands is unique (\cite{At56}) we get
\begin{equation*}
S^2\mathcal{G}=\mathcal{Q}_1(o)  \oplus \mathcal{Q}_2(o) \oplus \mathcal{Q}_3(o),
\end{equation*}
hence
\begin{equation*}
\begin{split}
S^3 \mathcal{G}\oplus \mathcal{G}(o) & = S^2\mathcal{G} \otimes \mathcal{G} = \mathcal{G} \otimes \mathcal{Q}_1(o)  \oplus \mathcal{G} \otimes \mathcal{Q}_2(o) \oplus \mathcal{G} \otimes \mathcal{Q}_3(o) \\
& = \mathcal{G}(o)  \oplus \mathcal{G}(o)  \oplus \mathcal{G}(o).
\end{split}
\end{equation*}
Therefore $S^3
\mathcal{G}=\mathcal{G}(o) \oplus \mathcal{G}(o)$ and by
straightforward computations one obtains
\begin{equation} \label{eq.S3.FF.ample}
S^3 \FF \otimes \bigwedge^2 \FF^{\vee} = \mathcal{F} \oplus
\mathcal{F}.
\end{equation}
Now the claim follows from \eqref{eq.H0.F.ample} and
\eqref{eq.S3.FF.ample}.
\end{proof}

\begin{cor} \label{cor.simple}
Assume that $\mathcal{L}$ is a product polarization and that $\FF$
is simple, and let $f \colon X \to A$ be the triple cover defined by
a general section $\eta \in H^0(A, \, S^3 \FF \otimes \bigwedge^2
\FF^{\vee})$. Then the variety $X$ is non-normal, and its
normalization $X^{\nu}$ is a properly elliptic surface with
$p_g(X^{\nu})=2$, $q(X^{\nu})=3$.
\end{cor}
\begin{proof}
Proposition \ref{prop.no.product.ample} shows that every section of
 $S^3 \FF \otimes \bigwedge^2 \FF^{\vee}$ vanishes along the curve
 $F_q$; this implies that $X$ is singular along
 $f^{-1}(F_q)$, in particular $X$ is non-normal. The composition of
 $f \colon X \to A$ with the normalization map is a triple cover
$f^{\nu} \colon X^{\nu} \to A$, whose Tschirnhausen bundle is
 $\mathcal{E}^{\nu}: =\mathcal{F}(-F_q)^{\vee}$. Since
 $\bigwedge^2\EE^{\nu}=\oo_A(-E_o)$, the morphism $f^{\nu}$
 is branched over a divisor belonging to the linear system $|2E_o|$,
 hence $X^{\nu}$ contains an elliptic fibration. Moreover
$c_1^2(\EE^{\nu})=0$, $c_2(\EE^{\nu})=0$ and a straightforward
 computation using $\FF(-F_q)=\pi_F^* \mathcal{G}$ and Leray
 spectral sequence yields
\begin{equation*}
h^0(A, \, \EE^{\nu})=0, \quad h^1(A, \, \EE^{\nu})=1, \quad h^2(A,
\, \EE^{\nu})=1.
\end{equation*}
Therefore Proposition \ref{prop.invariants} implies
$p_g(X^{\nu})=2$, $q(X^{\nu})=3$ and $K_{X^{\nu}}^2=0$, hence
$X^{\nu}$ is a properly elliptic surface.
\end{proof}

\section{Surfaces with $p_g=q=2$, $K^2_S=5$ and Albanese map of degree
$3$}\label{sec.degree.3}

\subsection{The triple cover construction} \label{subsec.ch}
The first example of a surface $S$ of general type with $p_g=q=2$
and $K_S^2=5$ was given by Chen and Hacon in \cite{CH06}, as a
triple cover of an abelian surface. In order to fix our notation,
let us recall their construction.

Let $(A, \, \mL)$ be a $(1, \,2)$-polarized abelian surface, and
assume that $\mL$ is a general, symmetric polarization. Since
\begin{equation*}
h^0(A, \, \mL \otimes \mathcal{Q} )=2, \quad h^1(A, \, \mL \otimes
\mathcal{Q})=0, \quad h^2(A, \, \mL \otimes \mathcal{Q})=0
\end{equation*}
for all $\mathcal{Q} \in \textrm{Pic}^0(A)$, the line bundle
$\mL^{-1}$ satisfies IT of index $2$. Then its Fourier-Mukai
transform $\mathcal{F}:= \widehat{\mL^{-1}}$ is a rank $2$ vector
bundle on $\wA$ which satisfies IT of index $0$, see \cite[Theorem
14.2.2]{BL04}. Let us consider the isogeny
\begin{equation*}
\phi:=\phi_{\mL^{-1}}\colon A \longrightarrow \wA,
\end{equation*}
whose kernel is $K(\mL^{-1})=K(\mL)$; then by \cite[Proposition 3.11]{Mu81}
we have
\begin{equation}\label{eq.mukai}
\phi^* \FF = \mL \oplus \mL.
\end{equation}

\begin{prop} \label{prop.coh.S3}
The vector bundle $S^3\FF \otimes \bigwedge^2 \FF^{\vee}$ satisfies
\begin{equation*}
h^0(\wA, \, S^3\FF \otimes \bigwedge^2 \FF^{\vee})=2, \; \; h^1(\wA,
\, S^3\FF \otimes \bigwedge^2 \FF^{\vee})=0, \; \; h^2(\wA, \,
S^3\FF \otimes \bigwedge^2 \FF^{\vee})=0.
\end{equation*}
\end{prop}
\begin{proof}
We could use Proposition \ref{prop.S3}, but we prefer a different
argument exploiting the isogeny $\phi$. Since $\chi(\wA, \, S^3\FF
\otimes \bigwedge^2 \FF^{\vee})=2$, it is sufficient to show that
$h^1(\wA, \, S^3\FF \otimes \bigwedge^2 \FF^{\vee})=h^2(\wA, \,
S^3\FF \otimes \bigwedge^2 \FF^{\vee})=0$. Since $\phi$ is a finite
map, we obtain
\begin{equation*}
H^i(\wA, \, S^3\FF \otimes \bigwedge^2 \FF^{\vee}) \cong
\phi^*H^i(\wA, \, S^3\FF \otimes \bigwedge^2 \FF^{\vee}) \subseteq
H^i(A, \, \phi^*(S^3\FF \otimes \bigwedge^2 \FF^{\vee}))
\end{equation*}
for all $i=0, \, 1, \, 2$. On the other hand, \eqref{eq.mukai}
yields
\begin{equation*}\label{eq.coh.split}
H^i(A, \, \phi^*(S^3\FF \otimes \bigwedge^2 \FF^{\vee}))
=H^i(A, \, \mL)^{\oplus 4},
\end{equation*}
so the claim follows.
\end{proof}

By Theorem \ref{teo.miranda} there is a $2$-dimensional family of
triple covers $\hat{f}\colon \widehat{X} \rightarrow \wA$ with
Tschirnhausen bundle $\EE=\FF^{\vee}$. We have the commutative
diagram
\begin{equation} \label{dia.projs}
\begin{xy}
\xymatrix{%%
X=A \times_{\wA} X  \ar[d]^{f} \ar[rr]^{\psi} & & \widehat{X} \ar[d]^{\widehat{f}} \\
A   \ar[rr]^{\phi} & & \widehat{A}  \\
  }
\end{xy}
\end{equation}
where $\psi\colon X \rightarrow \widehat{X}$ is a quadruple
\'etale cover and $f\colon X \rightarrow A$ is a triple cover
determined by a section of
\begin{equation*}\phi^*H^0(\wA, \, S^3\FF \otimes \bigwedge^2 \FF^{\vee}) \subset
H^0(A, \, \phi^*(S^3\FF \otimes \bigwedge^2 \FF^{\vee}))=H^0(A, \,
\mL)^{\oplus 4}.
\end{equation*}
By \cite[Chapter 6]{BL04} there exists a canonical Schr\"odinger
representation of the Heisenberg group
$\mathcal{H}_2$ on $H^0(A, \, \mL)$, where the latter space is identified
with the vector space $\CC(\ZZ/2\ZZ)$ of all complex valued
function on the finite group $\ZZ/2\ZZ$.

Following \cite[Section 2]{CH06} we can identify the $2$-dimensional
subspace of $H^0(A, \, \phi^*(S^3\FF \otimes \bigwedge^2
\FF^{\vee}))$ corresponding to $\phi^*H^0(\wA, \, S^3\FF \otimes
\bigwedge^2 \FF^{\vee})$ with
\begin{equation} \label{eq.sections}
\{(s x, \, t y, \, -t x, \,-s y) \, | \, s, \, t \in \CC \}
\subset H^0(A, \, \mL)^{\oplus 4},
\end{equation}
where $x$, $y \in H^0(A, \, \mL)$ form the canonical basis
induced by the characteristic functions of $0$ and $1$ in
$\CC(\ZZ/2\ZZ)$. By \cite{M85}, we can construct the triple cover
$f\colon X \rightarrow A$ using the data
\begin{equation} \label{eq.data.triple}
a=s x, \; \; b=t y, \; \; c=-t x, \; \; d=-s y.
\end{equation}
Over an affine open subset $U$ of $A$ the surface $X$ is defined
in $U \times \mathbb{A}^2$ by the determinantal equations
\begin{equation}\label{eq.determinant}
\textrm{rank} \, \left(\begin{array}{ccc}
z+a & w-2d & c \\
b & z-2a & w+d
\end{array} \right) \leq 1,
\end{equation}
where $w$, $z$ are coordinates in $\mathbb{A}^2$.
Moreover, the branch locus $D$ of $f\colon X \rightarrow A$ is given by
\begin{equation}\label{eq.branch}
D=(t^2-s^2)^2 x^2 y^2-4(s^2 x^2+st y^2)(s^2 y^2+st x^2) \in H^0(A,
\, \mL^4).
\end{equation}
This corresponds to a divisor $D_1+D_2+D_3+D_4$ with $D_i \in
|L|$; moreover the set $\{D_1, \ldots ,D_4\}$ is an orbit for
the action of $K(\mL)$ on $|L|$. For a general choice of $s, \, t$,
the $D_i$ are all smooth, so the singularity of $D$ are four
ordinary quadruple points at $e_0$, $e_1$, $e_2$, $e_3$. Over
these points $f \colon X \to A$ is totally ramified and $X$ has
four singularities of type $\frac{1}{3}(1, \,1)$. Blowing up these
points and the base points of $|L|$ we obtain a smooth triple
cover $\tilde{f} \colon \widetilde{X} \to \widetilde{A}$, which is
actually the canonical resolution of singularities of $X$, see
Proposition \ref{prop.can.ris}. Let $\{E_i\}_{i=1, \ldots,4}$ be
the exceptional divisor in $\widetilde{X}$ and $\{R_i\}_{i=1,
\ldots, 4}$ be the proper transform of the $D_i$ in
$\widetilde{X}$. Then $E_i^2=-3$, $E_iE_j=0$ for $i \neq j$,
$R_iR_j=0$ and $R_iE_j=1$ for all $i$, $j$. Since
\begin{equation*}
K_{\widetilde{X}}= \sum_{i=1}^4 R_i + \sum_{i=1}^4 E_i,
\end{equation*}
we obtain $K_{\widetilde{X}}^2=20$. Moreover $X$ has only rational
singularities, so if $\tilde{\sigma} \colon \widetilde{X} \to X$
is the resolution map we have $R^1 \tilde{\sigma}_*
\oo_{\widetilde{X}}= \oo_X$; therefore
\begin{equation*}
\begin{split}
p_g(\widetilde{X})&=h^2(\widetilde{X}, \, \oo_{\widetilde{X}})=
h^2(X, \, \oo_X)=h^2(A, \oo_A)+ 2h^2(A, \, \mL^{-1})=5; \\
q(\widetilde{X})&=h^1(\widetilde{X}, \, \oo_{\widetilde{X}})=
h^1(X, \, \oo_X)=h^1(A, \oo_A)+ 2h^1(A, \, \mL^{-1})=2.
\end{split}
\end{equation*}
This shows that $\chi(\widetilde{X}, \, \oo_{\widetilde{X}})=4$.
Now let $S$ be the canonical resolution of singularities of
$\widehat{X}$; then $K_S$ is ample and $\wA= \textrm{Alb}(S)$.
Since there is a quadruple,
\'etale cover $\tilde{\psi} \colon \widetilde{X} \lr S$ induced by
$\psi \colon X \to \widehat{X}$, the invariants of
$S$ are
\begin{equation*}
p_g(S)=q(S)=2, \quad K_S^2=5.
\end{equation*}
\begin{rem}
Both $X$ and $\widehat{X}$ only contain singular points of type
$\frac{1}{3}(1, \, 1)$, which are
 negligible singularities, see Example \ref{ex.1}.
 Hence we could compute the invariants of both
 $\widetilde{X}$ and
 $S$ by directly using Proposition \ref{prop.invariants}.
\end{rem}

\subsection{The product-quotient construction} \label{subsec.pe}
In \cite{Pe09} it is shown that there exists precisely one family
of surfaces with $p_g=q=2$ and $K_S^2=5$ which contain an
isotrivial fibration. Now we briefly explain how this family is
obtained, referring the reader to \cite{Pe09} for further details.

By using the Riemann Existence Theorem, one can construct two smooth
curves $C_1$, $C_2$ of genus $3$ which admit an action of the finite
group $S_3$, such that the $2$-cycles act without fixed points,
whereas the cyclic subgroup generated by the $3$-cycles has exactly
two fixed points. Then $E_i := C_i /S_3$ is a smooth elliptic curve
and the Galois cover $ C_i \to E_i$ is branched in exactly one point
with branching number $3$. Now let us consider the quotient
$\widehat{X}:= (C_1 \times C_2)/S_3$, where $S_3$ acts diagonally on
the product. Then $\widehat{X}$ contains precisely two cyclic
quotient singularities and, since the $3$-cycles are conjugated in
$S_3$, it is not difficult to show that one singularity is of type
$\frac{1}{3}(1, \,1)$ whereas the other is of type $\frac{1}{3}(1,
\, 2)$. Let $S \to \widehat{X}$ be the minimal resolution of
singularities of $\widehat{X}$; then $S$ is a minimal surface of
general type with $p_g=q=2$ and $K_S^2=5$; notice that $K_S$ is
\emph{not} ample. The surface $S$ admits two isotrivial fibrations
$S \to E_i$, which are induced by the two natural projections of
$C_1 \times C_2$.
\begin{figure}[H]
\begin{center}
\includegraphics*[totalheight=7 cm]{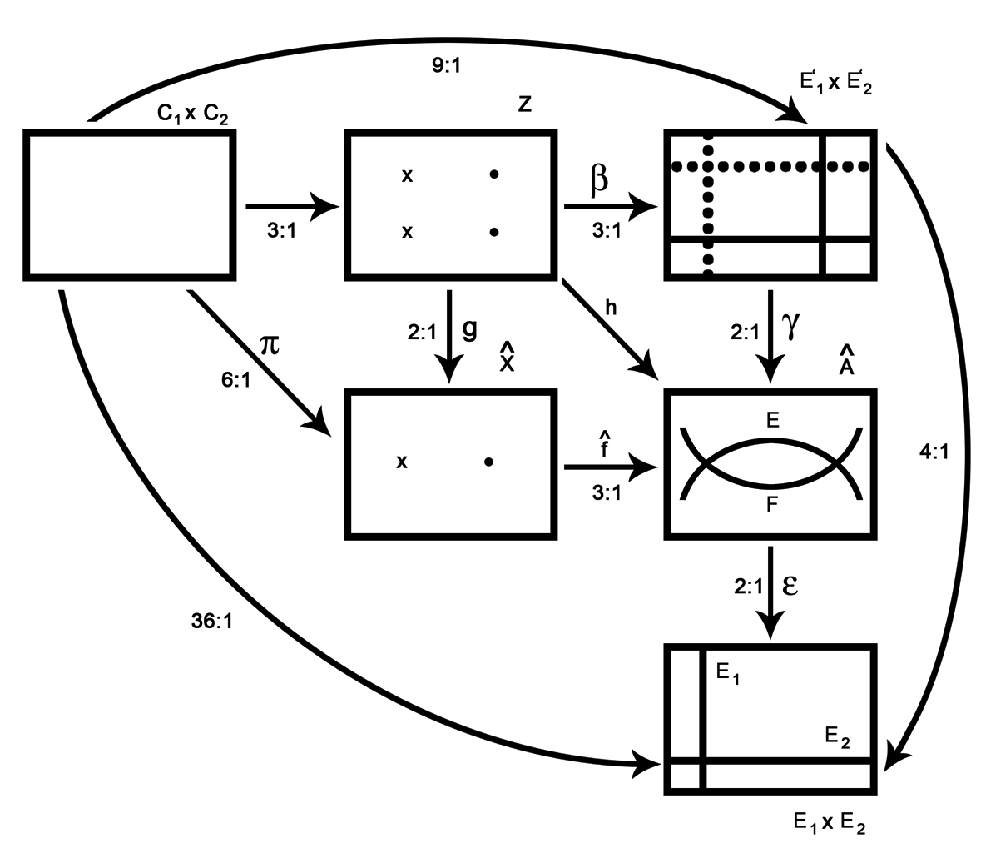}
\end{center}
\caption{The product-quotient construction} \label{figure-penegini}
\end{figure}
The Albanese variety $\widehat{A}$ of $S$ is an
\'etale double cover of $E_1 \times E_2$; it is actually a $(1,
\,2)$-polarized abelian variety, whose polarization $\mL$ is of
special type. The Albanese map $\alpha \colon S \to \widehat{A}$ is
totally ramified, and its reduced branch locus
$\Delta_{\textrm{red}}=E+F$ is a
 curve of type $(c)$ in Proposition \ref{prop.barth.class}, having
 one of its nodes in $\hat{o}$. It is clear that the two singular
 points of $\widehat{X}$ lie precisely over the two nodes of
 $\Delta_{\textrm{red}}$. In particular $\widehat{X}$ has only
 negligible singularities, see Example \ref{ex.3}.
This construction is summarized in Figure \ref{figure-penegini}.

There $\pi \colon C_1 \times C_2 \to \widehat{X}$ is induced by the
the diagonal action of $S_3$ on $C_1 \times C_2$, while
$\hat{f} \colon \widehat{X} \to \wA$ is the Stein factorization of
the Albanese map $\alpha \colon S \to \wA$. Since the diagonal
subgroup is not normal in $S_3 \times S_3$, it follows that
$\hat{f}$ is not a Galois cover; let $h \colon Z \rightarrow \wA$
be its Galois closure, which has Galois group $S_3$. The
surface $Z$ is isomorphic to the diagonal quotient $(C_1 \times
C_2)/(\mathbb{Z}/3\mathbb{Z})$, where $\mathbb{Z}/3 \mathbb{Z}$ is
 the subgroup of $S_3$ generated by the $3$-cycles; therefore $Z$ has four
 singular points coming
from the four points with non-trivial stabilizer on $C_1 \times
C_2$. More precisely,
\begin{equation*}
\textrm{Sing}(Z)=2 \times \frac{1}{3}(1, \,1) + 2 \times
\frac{1}{3}(1, \,2).
\end{equation*}
In addition, the cover $C_i \rightarrow E_i$ factors through the cover
$C_i \rightarrow E'_i:=C_i/(\mathbb{Z}/3\mathbb{Z})$, where $E'_i$
is an elliptic curve isogenous to $E_i$; this induces the cover
$C_1 \times C_2 \rightarrow (C_1 \times C_2)/
(\mathbb{Z}/3\mathbb{Z})^2=E'_1 \times E'_2$, which clearly
factors through $Z$. Observe that also the cover $\pi \colon C_1
\times C_2 \rightarrow \widehat{X}$ factors through $Z$. Finally
the composition $\epsilon \circ \gamma \colon E_1 \times E_2 \to
E_1' \times E_2'$ is a $(\mathbb{Z}/2\mathbb{Z})^2$-cover, which
factors through $\wA$. Using the commutativity of the diagrams in
Figure \ref{figure-penegini} and the theory of abelian covers
developed in \cite{Pa91}, one can check, looking at the building
data of $\beta\colon Z \rightarrow  E'_1 \times E'_2$ and
$\gamma\colon E'_1 \times E'_2 \rightarrow \wA$, that the
Tschirnhausen bundle $\EE$ of $f\colon \widehat{X} \rightarrow
\wA$ satisfies $\bigwedge^2 \EE^{\vee}=\mL \otimes \mathcal{Q}$,
where $\mathcal{Q}$ is a
non-trivial, $2$-torsion line bundle. This is a particular case of a
more general situation, see Proposition \ref{prop.no.Galois.1}.

\section{Chen-Hacon surfaces} \label{sec.CH}

In this section we will generalize the triple cover construction
described in Subsection \ref{subsec.ch}. In fact, since we want to
be able to ``take the limit" of a $1$-parameter family of surfaces
obtained in that way, we shall drop the assumptions that $\mL$ is
a general polarization and that $s$ and $t$ are general complex
numbers. Among other results, we will show that the
product-quotient surface described in Subsection \ref{subsec.pe}
can be also obtained as a specialization of Chen-Hacon's example,
see Corollary \ref{cor.peneg.}.

Let us start with the following

\begin{defin} \label{def.gch}
Let $S$ be a minimal surface of general type with $p_g=q=2$ such
that its Albanese map $\alpha \colon S \to
\widehat{A}:=\emph{Alb}(S)$ is a generically finite morphism of
degree $3$ onto an abelian surface $\widehat{A}$. Let
\begin{equation}\label{eq.stein.alb}
S \stackrel{p}{\lr} \widehat{X} \stackrel{\hat{f}}{\lr} \widehat{A}
\end{equation}
be the Stein factorization of $\alpha$, and $\mathcal{F}^{\vee}$ be
the Tschirnhausen bundle associated with the triple cover $\hat{f}$.
We say that $S$ a \emph{Chen-Hacon surface} if there exist a
polarization $\mL$ of type $(1, \,2)$ on
$A=\emph{Pic}^0(\widehat{A})$ such that $\FF =
\widehat{\mathcal{L}^{-1}}$.
\end{defin}

\begin{rem} \label{rem.min}
Since $\wA$ is an abelian variety and $\hat{f}$ is a finite map, it
follows that $p$ contracts all rational curves in $S$. The surface
$S$ is the minimal resolution of singularities of $\widehat{X}$ but
it is, in general, different from the canonical resolution
$\widetilde{X}$ described in Proposition \ref{prop.can.ris}. For
instance, in Example \ref{ex.2} the surface $\widetilde{X}$ contains
a $(-1)$-curve.
\end{rem}

The line bundle $\mathcal{L}$ is a IT-sheaf of index $0$, so by
\cite[Theorem 14.2.2]{BL04} and \cite[Proposition 14.4.3]{BL04} we
have
\begin{equation} \label{eq.FF-s}
h^0(\wA, \, \FF)=1, \quad h^1(\wA, \, \FF)=0, \quad h^2(\wA, \,
\FF)=0, \quad \det \FF=\mathcal{L}_{\delta},
\end{equation}
where $\mathcal{L}_{\delta}:=\oo_{\wA}(L_{\delta})$ is the dual
polarization of $\mathcal{L}$. Therefore $\FF$ belongs to the family
of bundles studied in Section \ref{sec.vec}.

\begin{prop} \label{prop.CH.no.decomposable}
Let $S$ be a Chen-Hacon surface. Then $\mathcal{F}$ is
indecomposable.
\end{prop}
\begin{proof}
Since $\mathcal{L}$ is a non-degenerate line bundle, by
\cite[Corollary 14.3.10]{BL04} it follows that $\FF$ is
$\mathcal{H}$-semistable with respect to any polarization
$\mathcal{H}$. Now the claim follows from Proposition
\ref{prop.decomp.unstable}. Alternatively, one could also remark that since $\mathcal{L}^{-1}$ is indecomposable the same must be true for its Fourier-Mukai transform $\mathcal{F}$.
\end{proof}

\begin{prop} \label{prop.CH.no.product}
Let $S$ be a Chen-Hacon surface. Then $\mathcal{L}$ is not a product
polarization.
\end{prop}
\begin{proof}
$\mathcal{L}$ is a product polarization if and only if
$\mathcal{L}_{\delta}$ is a product polarization. If
$\mathcal{L}_{\delta}$ were of product type, then $\widehat{X}$
would not be a surface of general type (see Corollaries
\ref{cor.nonsimple} and \ref{cor.simple}), contradiction.
\end{proof}

Since $\mathcal{L}$ is not a product polarization, we may use the
results of Subsection \ref{sec.ab.surf}. Moreover, for any
Chen-Hacon surface $S$ we can consider its associated diagram
\eqref{dia.projs}. Being the morphism $\psi$ \'etale, $X$ is
nonsingular in codimension one if and only if the same holds for
$\widehat{X}$. Similarly, $f$ is totally ramified if and only if
$\hat{f}$ is totally ramified.

\begin{prop} \label{prop.iso.sing}
The following holds:
\begin{itemize}
\item[$\boldsymbol{(i)}$] $X$ has only isolated singularities unless
$t=0$ or $t^2-9s^2=0$.
\item[$\boldsymbol{(ii)}$] If $t=0$ or $t^2-9s^2=0$, then $X$ has non-isolated singularities.
Moreover, if $\nu \colon X^{\nu} \to X$ is the normalization map, then the composition
$f \circ \nu \colon X^{\nu} \to A$ is an \'etale triple cover. Therefore,
in this case $X$ is
not a surface of general type.
\end{itemize}
\end{prop}
\begin{proof}
$\boldsymbol{(i)}$ A local computation as in \cite[Claim 2]{CH06} shows that,
if $t \neq 0$ and $t^2 \neq 9s^2$, above a neighborhood of
any of the base points of $|L|$ the equations \eqref{eq.determinant} define
a cone over a twisted cubic, hence an isolated singularity of type $\frac{1}{3}(1,1)$.

$\boldsymbol{(ii)}$ We can assume $t=0$, since
 the proof in the other cases is the same. Looking at \eqref{eq.determinant},
 we see that in a neighborhood of any of the
 base points $e_0, \, e_1, \, e_2, \, e_3$,
 the surface $X$ is defined in $\mathbb{A}^4$ by
\begin{equation*}
(x+z)(2x-z)=0, \; \; (2y+w)(y-w)=0, \; \; (x+z)(y-w)=0,
\end{equation*}
 and it is straightforward to see that these
equations define the union of three $2$-planes intersecting along
two lines. This shows that $X$ contains non-isolated
singularities. The normalization map $\nu \colon X^{\nu} \to X$
can be computed by using the Computer Algebra
System \verb|Singular|, see \cite{SING}. It turns out that
 $X^{\nu}$ is locally given by three mutually disjoint $2$-planes
 in $\mathbb{A}^5$; moreover, for each of these planes the
 projection onto the first two coordinates of $\mathbb{A}^5$ is
 an isomorphism. In the global picture this means that $X^{\nu}$ is smooth and
 $f \circ \nu \colon X^{\nu}
\to A$ is an \'etale triple cover.
\end{proof}

\begin{rem} \label{rem.red}
In Proposition \ref{reducible} we will show that if $t=0$ or $t^2-9s^2=0$
then $\widehat{X}$ (and hence $X$) is a reducible surface.
\end{rem}

\begin{prop} \label{prop.total.ram}
Assume that $X$ has only isolated singularities. Then the
following holds:
\begin{itemize}
\item[$\boldsymbol{(i)}$] $f \colon X \rightarrow A$ is totally
ramified if and only if $s=t$, $s=-t$ or $s=0.$
\item[$\boldsymbol{(ii)}$] $f \colon X \to A$ is totally ramified if
and only if
\begin{itemize}
\item[$\boldsymbol{(iia)}$] either $D=2D_1+2D_2$, where $D_1,
\,D_2 \in |L|$ are distinct, smooth hyperelliptic curves belonging
to the same $K(\mL)$-orbit, or
\item[$\boldsymbol{(iib)}$] $\mL$ is a special polarization and
$D=2(E+F)+2(E'+F')$, where $E+F$ and $E'+F'$ are as in
Proposition \emph{\ref{prop.lin.M.sing}} $(b)$.
\end{itemize}
\end{itemize}
\end{prop}
\begin{proof}
$\boldsymbol{(i)}$ The triple cover $f \colon X \to A$ is totally
ramified if and only if the discriminant of the polynomial
defining $D$ in \eqref{eq.branch} vanishes. This happens exactly
for $s=0$, $t=0$, $s=t$, $s=-t$, $t=3s$, $t=-3s$. Since we are
assuming that $X$ has isolated singularities, the only acceptable
values are $s=t$, $s=-t$ and $s=0$ (see Proposition \ref{prop.iso.sing}).

$\boldsymbol{(ii)}$ The triple cover $f
\colon X \to A$ is totally ramified if and only if $D=2D'$ for
some effective divisor $D'$. Since the four curves $D_i$ form an
orbit for the action of $K(\mL)$ on $|L|$, this is equivalent to
say that the $D_i$ have non-trivial stabilizer. Now the assertion
follows from Proposition \ref{prop.hyp.stab}.
\end{proof}

\begin{prop} \label{prop.total.ram.2}
Assume that $\widehat{X}$ has only isolated singularities. Then
$\widehat{X}$ always contains a singular point of type
$\frac{1}{3}(1,1)$, lying over $\hat{o} \in \widehat{A}$. Moreover,
this point is the unique singular point of $\widehat{X}$, unless:
\begin{itemize}
\item[$\boldsymbol{(i)}$] one of the $D_i$ is an irreducible,
nodal curve$;$ in this case $\widehat{X}$ also contains a singular
point of type $\frac{1}{2}(1, \,1);$
\item[$\boldsymbol{(ii)}$]
$\mL$ is a special polarization and we are in case $(iib)$ of
Proposition \emph{\ref{prop.total.ram}}. Then
$\hat{f} \colon \widehat{X} \to \wA$ is
totally ramified over the image in $\widehat{A}$ of the divisor
$E+F+E'+F'$, which is a curve isomorphic to $E+F$ and having a
node at $\hat{o}$. In this case $\widehat{X}$ also contains a
singular point of type $\frac{1}{3}(1, \, 2)$.
\end{itemize}
\end{prop}
\begin{proof}
Since there exists an \'etale morphism $\psi \colon X \to
\widehat{X}$, it is sufficient to
 analyze the triple cover $f \colon X \to A$.
If all divisors $D_i$ are smooth, then the only singularities of $X$
are the four points of type $\frac{1}{3}(1, \, 1)$ lying over the
base points of $|L|$. If one of the $D_i$ is an irreducible, nodal
curve, then all the $D_i$ are so, because they form a single
$K(\mL)$-orbit, see Proposition \ref{prop.lin.M.sing} and Remark \ref{Yoshihara}.
In this case $X$ also contains
four points of type $\frac{1}{2}(1, \,1)$, which are identified by
$\psi$ to a unique point of type $\frac{1}{2}(1, \, 1)$ in
$\widehat{X}$; this
 yields $(i)$. Finally, if $\mL$ is a special polarization and
$D=2(E+F)+2(E'+F')$, then locally around any of the four points
$p, \, q, \, r, \,  s$ the equation of $X$ can be written as
$z^3=xy$, so they give singularities of type $\frac{1}{3}(1,
\,2)$. The morphism $\psi$ identifies $E$ with $E'$ and $F$ with
$F'$. Then $\hat{f} \colon \widehat{X} \to \widehat{A}$ is totally
ramified and its reduced branch locus is isomorphic to $E+F$, in
particular it has two nodes. One of these nodes is at $\hat{o}$
and it gives the singular point of type $\frac{1}{3}(1, \,1)$; the
second one gives instead a singular point of type
$\frac{1}{3}(1,2)$. This is case $(ii)$.
\end{proof}

In the sequel we will denote by $\Delta$ the branch locus of
$\hat{f} \colon \widehat{X} \to \widehat{A}$.  By construction, it
is precisely the image of $D$ via $\phi \colon A \to \wA$. It
follows that $\Delta$ always has a point of multiplicity $4$ at
$\hat{o} \in \widehat{A}$. More precisely, we have the following

\begin{prop} \label{prop.quadruple}
The branch locus $\Delta$ belongs precisely to one of the following
types:
\begin{itemize}
\item[$\boldsymbol{(a)}$] $\Delta$ is reduced and its only
singularity is an ordinary quadruple point at $\hat{o};$ in this
case $\emph{Sing}(\widehat{X})=\frac{1}{3}(1, \, 1).$
\item[$\boldsymbol{(b)}$] $\Delta$ is reduced and its only
singularities are an ordinary quadruple point at $\hat{o}$ and an
ordinary double point; in this case
$\emph{Sing}(\widehat{X})=\frac{1}{3}(1, \, 1) + \frac{1}{2}(1, \,
1).$ \item[$\boldsymbol{(c)}$] $\Delta=2 \Delta_{\emph{red}}$,
where $\Delta_{\emph{red}}$ is an irreducible curve whose unique
singularity is an ordinary double point at $\hat{o};$ in this case
$\emph{Sing}(\widehat{X})=\frac{1}{3}(1, \, 1).$
\item[$\boldsymbol{(d)}$] $\Delta=2 \Delta_{\emph{red}}$ and
$\Delta_{\emph{red}}=E+F$, where $E$, $F$ are elliptic curves such
that $EF=2$ and $\hat{o} \in E \cap F;$
in this case $\emph{Sing}(\widehat{X})=\frac{1}{3}(1,
\, 1) + \frac{1}{3}(1, \, 2).$
\end{itemize}
The canonical divisor $K_S$ is ample if and only if we are either
in case $(a)$ or in case $(c)$.
\end{prop}
\begin{proof}
Case $(a)$ corresponds to the general situation. Case $(b)$
corresponds to Proposition \ref{prop.total.ram.2}, $(i)$. Case
$(c)$ corresponds to Proposition \ref{prop.total.ram}, $(iia)$.
Finally, Case $(d)$ corresponds to Proposition
\ref{prop.total.ram.2}, $(ii)$ or, equivalently, to Proposition
\ref{prop.total.ram}, $(iib)$.
\end{proof}

\begin{rem} \label{rem.dual.pol}
The equation of $\Delta$ is given by a
non-zero element in $H^0(\wA, \, \mL_{\delta}^2 \otimes
\mathcal{I}^4_{\hat{o}})$,
 where $\mL_{\delta}$ is a $(1, \,2)$-polarization on $\wA$
which coincides, up to translations, with the dual polarization of
$\mL$, see \cite[Chapter 14]{BL04} (we cannot denote the dual polarization by
$\widehat{\mL}$, since this is the Fourier-Mukai transform of $\mL$).
Notice that the four cases in
Proposition \ref{prop.quadruple} correspond exactly to the ones in
Proposition \ref{prop.pencil.I4}.
\end{rem}

Summarizing the results obtained in this section, we have

\begin{prop} \label{prop.ch}
If $S$ is a  Chen-Hacon surface, then it is a minimal surface of
general type with $p_g=q=2$, $K_S^2=5$. Moreover $\widehat{X}$
contains at least one and at most two isolated, negligible
singularities, which belong to the the types described in Examples
\emph{\ref{ex.1}}, \emph{\ref{ex.2}}, \emph{\ref{ex.3}}.
In particular, $\widehat{X}$ is never smooth.
\end{prop}

\section{Characterization of Chen-Hacon surfaces} \label{sec.main.thm}

In this section we prove one of the key results
of the paper, namely the following converse of Proposition \ref{prop.ch}.

\begin{theo} \label{teo.ch}
Let $S$ be a minimal surface of general type with $p_g=q=2$, $K_S^2=5$
such that the Albanese map
$\alpha \colon S \to \widehat{A}:= \emph{Alb}(S)$
is a generically finite morphism of degree $3$. Let
\begin{equation*}
S \stackrel{p}{\lr} \widehat{X} \stackrel{\hat{f}}{\lr} \widehat{A}
\end{equation*}
be the Stein factorization of $\alpha$. If $\widehat{X}$ has at most
negligible singularities, then $S$ is a Chen-Hacon surface.
\end{theo}
The proof will be a consequence of Propositions \ref{prop.no.decomp} and \ref{prop.exist.L} below. Let
$\mathcal{E}$ be the Tschirnhausen bundle of the triple cover
$\hat{f} \colon \widehat{X} \to \widehat{A}$. Since by assumption
$\widehat{X}$ has at most negligible singularities, Proposition
\ref{prop.invariants} implies
\begin{equation} \label{eq.inv.E}
\begin{split}
h^0(\wA, \, \EE)& =0, \quad h^1(\wA, \, \EE)=0, \quad h^2(\wA, \,
\EE)=1; \\
 c_1^2(\EE)&=4, \quad c_2(\EE)=1.
\end{split}
\end{equation}
In particular, $\bigwedge^2 \EE^{\vee}$ yields a polarization of
type $(1, \,2)$ on $\wA$; let us denote it by
$\mL_{\delta}=\oo_A(L_{\delta})$. Setting $\FF:=\EE^{\vee}$, we have
\begin{equation*}
h^0(\wA, \, \FF)=1, \quad h^1(\wA, \, \FF)=0, \quad h^2(\wA, \,
\FF)=0, \quad \det \FF = \mL_{\delta},
\end{equation*}

that is $\FF$ belongs to the family of vector bundles studied in
Section \ref{sec.vec}.

\begin{prop} \label{prop.no.decomp}
$\mathcal{F}$ is an indecomposable vector bundle.
\end{prop}
\begin{proof}
Assume that $\mathcal{F}$ is decomposable. Then there exists a
line bundle $\mathcal{C}=\oo_{\wA}(C)$ such that
\begin{equation*}
\FF = \mathcal{C} \oplus ( \mathcal{C}^{-1} \otimes
\mathcal{L}_{\delta}).
\end{equation*}
Following \cite[Section 6]{M85}, we can construct $\hat{f} \colon
\widehat{X} \to \wA$ by using the data
\begin{equation*}
\begin{split}
a & \in H^0(\wA, \, \mathcal{C}), \\ b &  \in H^0(\wA, \,
\mathcal{C}^3 \otimes \mathcal{L}_{\delta}^{-1}), \\
c & \in H^0(\wA, \, \mathcal{C}^{-3} \otimes \mathcal{L}_{\delta}^2 ), \\
d & \in H^0(\wA, \, \mathcal{C}^{-1} \otimes
\mathcal{L}_{\delta}).
\end{split}
\end{equation*}
Moreover, being $\widehat{X}$ irreducible, $b$ and $c$ are both
non-zero.

Since $h^0(\wA, \, \mathcal{F})=1$, we may assume $h^0(\wA, \,
\mathcal{C})=1$ and $h^0(\wA, \,  \mathcal{C}^{-1} \otimes
\mathcal{L}_{\delta})=0$. Therefore there are two possibilities:
\begin{itemize}
\item[$(i)$] $C$ is an elliptic curve; \item[$(ii)$]
$C$ is a principal polarization.
\end{itemize}

In case $(i)$, we have $1=C(\Ld - C)=C \Ld$. Then $(3C- \Ld)\Ld=-1$,
so $3C-\Ld$ cannot be effective. This implies $b=0$, contradiction.

In case $(ii)$, the Index Theorem yields $8=C^2 \Ld^2 \leq (C
\Ld)^2$, so $C \Ld \geq 3$. It follows
\begin{equation*}
(-3C + 2 \Ld) \Ld = -3 C \Ld + 8 \leq -1,
\end{equation*}
hence $c=0$, contradiction.
\end{proof}

\begin{prop} \label{prop.no.pol.prod}
$\mathcal{L}_{\delta}$ is not a product polarization.
\end{prop}
\begin{proof}
By the results of Section \ref{sec.vec}, especially Corollaries
\ref{cor.nonsimple} and \ref{cor.simple}, if $\mathcal{L}_{\delta}$
were a product polarization then $\widehat{X}$ would be either a
reducible surface or a non-normal surface birational to a properly
elliptic surface, in particular it would not be a surface of general
type.
\end{proof}

\begin{prop} \label{prop.exist.L}
There exists a symmetric $(1,2)$-polarization $\mL$ on $A$ such that
\begin{equation*}
\widehat{\mL^{-1}}=\FF.
\end{equation*}
\end{prop}
\begin{proof}
Since $\mathcal{L}_{\delta}$ is not a product polarization
(Proposition \ref{prop.no.pol.prod}), it follows that $\FF$ is the
unique non-trivial extension
\begin{equation} \label{eq.F.in.wA}
0 \to \mathcal{O}_{\wA} \to \FF \to \mathcal{L}_{\delta} \otimes
\mathcal{I}_{\wo} \to 0,
\end{equation}
see Proposition \ref{prop.L.irred.F.simple}. Moreover,
$(-1)_{\wA}^* \FF = \FF$ and $\FF$ satisfies IT of index $0$
(Proposition \ref{prop.F.WIT}). Thus $\widehat{\FF}$ is a line
bundle on $A$ that we denote by $\mL^{-1}$;
 the sheaf $\mL$ satisfies IT of index $0$ too,
 see \cite[Theorem 14.2.2]{BL04}.
Therefore by \cite{Mu81} we get
\begin{equation*}
\widehat{\mL^{-1}}=\widehat{(\widehat{\FF})}=(-1)^*_{\wA}\FF=\FF.
\end{equation*}
Since $h^0(A, \, \mL)=\textrm{rank}(\FF)=2$, it follows that
$\mathcal{L}$ is a $(1, \, 2)$-polarization. Notice that
$\mathcal{L}$ coincides with the dual polarization of
$\mathcal{L}_{\delta}$, in particular it is not a product
polarization (see also Remark \ref{rem.dual.pol}).
\end{proof}
This completes the proof of Theorem \ref{teo.ch}.

\begin{rem} \label{rem.S3}
It is interesting to compare Proposition \ref{prop.S3} with
Proposition \ref{prop.2M}. In fact, an explicit isomorphism
$H^0(\widehat{A}, \, \mL_{\delta}^2 \otimes \mathcal{I}^3_{\wo})
\stackrel{\cong}{\to} H^0(\widehat{A}, \, \mL_{\delta}^2 \otimes
\mathcal{I}^4_{\wo})$ can be given by associating to every section
$\eta \in H^0(\widehat{A}, \, \mL_{\delta}^2 \otimes
\mathcal{I}^3_{\wo}) \cong H^0(\wA, \, S^3 \FF \otimes \bigwedge^2
\FF^{\vee})$ the equation defining the branch locus $\Delta$ of the
triple cover given by $\eta$, see again Remark \ref{rem.dual.pol}.
\end{rem}

An immediate consequence of Theorem \ref{teo.ch} is

\begin{cor} \label{cor.peneg.}
The isotrivially fibred surface constructed in \emph{\cite{Pe09}}, i.e., the
product-quotient surface of Subsection \emph{\ref{subsec.pe}}, is a
Chen-Hacon surface. More precisely, it corresponds to case $(ii)$ of
Proposition \emph{\ref{prop.total.ram.2}} or, equivalently, to case $(d)$
of Proposition \emph{\ref{prop.quadruple}}.
\end{cor}
\begin{proof}
The product-quotient surface contains only negligible singularities,
see Example \ref{ex.3}, so Theorem \ref{teo.ch} implies
that it is a Chen-Hacon surface. Since $\widehat{X}$ has one singularity of type
$\frac{1}{3}(1, \, 1)$ and one singularity
of type $\frac{1}{3}(1,\,2)$, looking at Proposition \ref{prop.total.ram.2}
we see that it  corresponds to case $(ii)$.
\end{proof}

The remainder of this section deals with some further properties of
Chen-Hacon surfaces.

\begin{prop} \label{prop.never.finite}
Let $S$ be a Chen-Hacon surface. Then $\alpha \colon S \to
\widehat{A}$ is never a finite morphism.
\end{prop}
\begin{proof}
By Proposition \ref{prop.quadruple}, $S$ always contains a
$(-3)$-curve, which is contracted by $\alpha$.
\end{proof}

\begin{prop} \label{prop.no.Galois.1}
Let $S$ be a  Chen-Hacon surface, and assume that $\hat{f} \colon
\widehat{X} \to \widehat{A}$ is totally ramified. Then
$\Delta_{\emph{red}}$ is linearly equivalent to $\Ld+Q$,
where $Q$ is a non-trivial, $2$-torsion divisor.
\end{prop}
\begin{proof}
By \cite[Proposition 4.7]{M85} the divisor $\Delta=2
\Delta_{\textrm{red}}$ is linearly equivalent to $2 \Ld$,
hence $\Delta_{\textrm{red}}$ is linearly equivalent to
$L_{\delta}+Q$, where $Q$ is a $2$-torsion divisor. On the other
hand, $\Delta_{\textrm{red}}$ is singular at $\hat{o}$ (Proposition
\ref{prop.quadruple}), so $Q$ is not trivial.
\end{proof}

\begin{prop} \label{prop.no.Galois.2}
Let $S$ be a Chen-Hacon surface. Then $\hat{f} \colon \widehat{X}
\to \widehat{A}$ is never a Galois cover.
\end{prop}
\begin{proof}
By \cite[Theorem 5.5]{TZ04} it follows that $\hat{f}$ is a Galois
cover if and only if it is totally ramified and the line bundle
$\bigwedge^2 \mathcal{F}$ is isomorphic to
$\mathcal{O}_{\wA}(\Delta_{\textrm{red}})$. Since $\bigwedge^2 \FF= \mL_{\delta}$,
this is excluded by Proposition
\ref{prop.no.Galois.1}.
\end{proof}

\begin{prop} \label{prop.no.pencils}
Let $S$ be a Chen-Hacon surface, and assume that $A$ is a simple
abelian surface. Then $S$ does not contain any pencil $p \colon S
\to B$ over a curve $B$ with $g(B) \geq 1$.
\end{prop}
\begin{proof}
Since $A$ is simple, the same is true for $\wA$. Then the set
$V^1(S):=\{\mathcal{Q} \in \textrm{Pic}^0(S) \; | \; h^1(S, \, \mathcal{Q}) > 0 \}$
cannot contain any component of positive dimension, and
$S$ does not admit any pencil over a curve $B$ with $g(B) \geq 2$,
see \cite[Theorem 2.6]{HP02}. If instead $g(B)=1$,
 the universal property of the Albanese
map yields a surjective morphism $\wA \to B$,
contradicting again the fact that $\wA$ is simple.
This concludes the proof.
\end{proof}

It would be very interesting to classify the possible
degenerations of Chen-Hacon surfaces; however, this problem is at
present far from being solved. The following result describes some
natural degenerations obtained by taking reducible triple covers.

\begin{prop} \label{reducible}
Let $\hat{f} \colon \widehat{X} \to \widehat{A}$ be the non-normal
triple cover corresponding to either $t=0$ or $t^2=9s^2$ $($see
Proposition \emph{\ref{prop.iso.sing}}$)$. Then $\widehat{X}$ is a
reducible surface. More precisely, there exists $i \in \{1, \,2,
\,3\}$ such that the section defining $\hat{f}$ is in the image of
the multiplication map
\begin{equation*}
H^0(\wA, \, S^2 \FF \otimes \bigwedge^2 \FF^{\vee} \otimes \mathcal{Q}_i)
\otimes H^0(\wA, \, \FF \otimes \mathcal{Q}_i) \lr H^0(\wA, \, S^3 \FF \otimes
\bigwedge^2 \FF^{\vee}),
\end{equation*}
where the $\mathcal{Q}_i$ are the non-trivial, $2$-torsion line
bundles on $\widehat{A}$ defined as in \eqref{eq.2-tors}.
\end{prop}
\begin{proof}
It is sufficient to show that
\begin{equation*}
h^0(\wA, \, S^2 \FF \otimes \bigwedge^2 \FF^{\vee} \otimes \mathcal{Q}_i) \neq
0 \; \; \textrm{and} \; \; h^0(\wA, \, \FF \otimes \mathcal{Q}_i) \neq 0
\end{equation*}
for $i=1, \,2, \,3$.
Tensoring \eqref{suc.def.f.triv} with $\mathcal{Q}_i$ and using \eqref{eq.Io-Io2} we
obtain
\begin{equation*}
h^0(\wA, \, \FF \otimes \mathcal{Q}_i)=h^0(\wA, \, \mL_{\delta} \otimes \mathcal{Q}_i
\otimes \mathcal{I}_{\wo} )=1.
\end{equation*}
On the other hand, Eagon-Northcott complex applied to
\eqref{eq.F.in.wA} gives
\begin{equation*}
0 \lr \FF \lr S^2 \FF  \lr \mL_{\delta}^2 \otimes
\mathcal{I}^2_{\hat{o}} \lr 0,
\end{equation*}
hence we obtain
\begin{equation} \label{suc.red}
0 \lr \FF \otimes \bigwedge^2 \FF^{\vee} \otimes \mathcal{Q}_i \lr
S^2 \FF \otimes \bigwedge^2 \FF^{\vee} \otimes \mathcal{Q}_i  \lr
\mL_{\delta} \otimes \mathcal{Q}_i \otimes \mathcal{I}^2_{\hat{o}}
\lr 0.
\end{equation}
Using $\FF \otimes \bigwedge^2 \FF^{\vee}=\FF^{\vee}$, Serre duality
and \eqref{eq.WIT} we deduce
\begin{equation*}
\begin{split}
h^0(\wA, \, \FF \otimes \bigwedge^2 \FF^{\vee} \otimes \mathcal{Q}_i) &
=h^0(\wA, \, \FF^{\vee}
\otimes \mathcal{Q}_i)=h^2(\wA, \, \FF \otimes \mathcal{Q}_i)=0, \\
h^1(\wA, \, \FF \otimes \bigwedge^2 \FF^{\vee} \otimes \mathcal{Q}_i) &
=h^1(\wA, \, \FF^{\vee} \otimes \mathcal{Q}_i)=h^1(\wA, \, \FF \otimes \mathcal{Q}_i)=0,
\end{split}
\end{equation*}
so by \eqref{eq.Io-Io2} we have
\begin{equation*}
h^0(\wA, \, S^2 \FF \otimes \bigwedge^2 \FF^{\vee} \otimes \mathcal{Q}_i)=
h^0(\wA, \, \mL_{\delta} \otimes \mathcal{Q}_i \otimes \mathcal{I}^2_{\hat{o}})=1.
\end{equation*}
This completes the proof.
\end{proof}

\begin{rem}
Further degenerations of Chen-Hacon surfaces could be obtained by
looking at the case where $\mathcal{L}_{\delta}$ becomes a product
polarization, see Corollaries \ref{cor.nonsimple} and
\ref{cor.simple}.
\end{rem}

We will now describe the canonical system $|K_S|$ of a Chen-Hacon
surface $S$, showing that it is composed with a rational pencil of
curves of genus $3$.

For the sake of simplicity, we will assume that $A$ is a simple
abelian surface. Let $\alpha \colon S \to \wA$ be the Albanese map
of $S$, let $\sigma \colon \BlA \to \wA$ be the blow-up of $\wA$ at
$\wo$, and let $\Lambda \subset \wA$ be the exceptional divisor.
Then there is an induced map $\beta \colon S \to \BlA$, which is a
flat triple cover. The branch locus of $\beta$ coincides with the
strict transform of the branch locus $\Delta$ of $\hat{f}$, so it
belongs to the strict transform of the pencil $\mathfrak{D}_{\delta}
\subset |2 L_{\delta}|$ given by $\mathbb{P}H^0(\wA, \,
\mL_{\delta}^2 \otimes \mathcal{I}_{\wo}^4)$. The general element in
this pencil is a smooth curve of genus $3$ and self-intersection
$0$, meeting $\Lambda$ in precisely four distinct points; so we have
a base-point free pencil $\hat{\varphi} \colon \BlA \to
\mathbb{P}^1$. The exceptional divisor $\Lambda$ is not in the
branch locus of $\beta$ and $\Xi:=\beta^*(\Lambda)$ is the unique
$(-3)$-curve in $S$. Considering the Stein factorization of the
composed map $S \stackrel{\beta} \to \BlA \stackrel{\hat{\varphi}}
\to \mathbb{P}^1$, and using Proposition \ref{prop.no.pencils}, we
obtain a commutative diagram
\begin{equation} \label{diag.can.sist}
\begin{xy}\xymatrix{
S   \ar[d]_{\varphi} \ar[rr]^{\beta} & & \BlA \ar[d]^{\hat{\varphi}} \\
\mathbb{P}^1   \ar[rr]^{b} & & \mathbb{P}^1,  \\
  }
\end{xy}
\end{equation}
where $b \colon \mathbb{P}^1 \to \mathbb{P}^1$ is a triple cover
simply branched on four points, corresponding to the branch
curve of $\beta$ and to the three double curves in $\mathfrak{D}_{\delta}$, see
Proposition \ref{prop.pencil.I4}.

\begin{prop} \label{prop.canonical}
Let $S$ be a Chen-Hacon surface. Then $|K_S|=\Xi + |\Phi|$, where
$\Phi$ is a smooth curve of genus $3$ which satisfies $h^0(S, \,
\oo_S(\Phi))=2$, $\Phi^2=0$, $\Xi \Phi=4$. It follows that $\varphi
\colon S \to \mathbb{P}^1$ coincides with the canonical map
$\varphi_{|K_S|}$ of $S$.
\end{prop}
\begin{proof}
The canonical divisor of $S$ is given by
\begin{equation*}
K_S= \beta^* K_{\BlA} + R = \Xi + R,
\end{equation*}
where $R$ is the ramification divisor of $\beta$. By diagram
\eqref{diag.can.sist} it follows that $R \in |\Phi|$, where
$|\Phi|$ is the pencil induced by $\varphi$.
The general element of $|\Phi|$ is
a smooth curve of genus $3$,
isomorphic to the strict transform of the
general element of $\mathfrak{D}_{\delta}$. Since
\begin{equation*}
2= h^0(S, \, \oo_S(K_S))=h^0(S, \, \oo_S(\Phi))=h^0(S, \,
\oo_S(K_S - \Xi)),
\end{equation*}
it follows that  $\Xi$ is contained in the fixed part of
$|K_S|$. The rest of the proof is now clear.
\end{proof}

%%%%%%%%%%%%%%%%%%%%%%%%%%%%%%%%%%%%%%%%%%%%%%%%%%%%%%%%%%%%%%%%%%%%%%%%%%%%%%%%%%%%%%%%%%%%
%%%%%%%%%%%%%%%%%%%%%%%%%%%%%%%%%%%%%%%%%%%%%%%%%%%%%%%%%%%%%%%%%%%%%%%%%%%%%%%%%%%%%%%%%%%%

\section{The moduli space} \label{sec.moduli}

The aim of this  section is to investigate the deformations of
Chen-Hacon surfaces. The first step is to embed $S$ in the projective bundle $\mathbb{P}(\FF)$ as a divisor \emph{containing a fibre}.

\begin{prop}\label{prop.embedding}
Let $S$ be a Chen-Hacon surface with ample canonical class; then there is an embedding $i\colon S \hookrightarrow \mathbb{P}(\FF)$ whose image is a smooth
divisor in the linear system $|3\xi - \pi^*L_{\delta}|$, where $\xi$
is the divisor class of $\oo_{\mP}(1)$ and $\pi \colon \mathbb{P} \to \wA$
is the natural projection. Moreover $i(S)$ contains the fibre
$\pi^{-1}(\wo)$ of $\mathbb{P}$; more precisely, $\pi^{-1}(\wo)$
coincides with the unique  $(-3)$-curve $\Xi$ of $S$.
\end{prop}
\begin{proof}
Let us consider again the blow-up $\sigma \colon \wA^{\sharp} \to
\wA$, with exceptional divisor $\Lambda \subset \wA^{\sharp}$, and
the flat triple cover $\beta\colon S \longrightarrow
\widehat{A}^{\sharp}$ described in the previous section. A
straightforward calculation shows that the Tschirnhausen bundle
associated to $\beta$ is
\begin{equation*}  \FF^{\sharp} = \sigma^*\FF \otimes \mathcal{O}_{\widehat{A}^{\sharp}}(-\Lambda)
\end{equation*}
and that we have a commutative diagram
 \begin{equation}
\begin{xy}\xymatrix{
 \mathbb{P}(\FF^{\sharp})  \ar[d]_{\pi^{\sharp}} \ar[rr]^{\tau} & &  \mathbb{P}(\FF) \ar[d]^{\pi} \\
\widehat{A}^{\sharp}   \ar[rr]^{\sigma} & & \widehat{A}.  \\
  }
\end{xy}
\end{equation}
Since $S$ is smooth and $\beta$ is flat, by \cite{CE96} there is an embedding $i^{\sharp}\colon S \hookrightarrow \mathbb{P}(\FF^{\sharp})$. Its image is a divisor in the linear system
$|3\xi^{\sharp}-(\pi^{\sharp})^{*}\det\FF^{\sharp}|$, where $\xi^{\sharp}$
is the divisor class of $\oo_{\mP(\FF^{\sharp})}(1)$ and $\pi^{\sharp} \colon \mathbb{P}(\FF^{\sharp}) \to \wA^{\sharp}$
is the projection. We have natural identifications
\begin{equation*}
\begin{split}
H^0\big(\mP(\FF^{\sharp}), \, 3\xi^{\sharp}-(\pi^{\sharp})^{*}\det\FF^{\sharp}\big) & \cong H^0\big(\widehat{A}^{\sharp}, \,S^3\FF^{\sharp} \otimes \bigwedge^2 (\FF^{\sharp})^{\vee}\big) \\
& \cong H^0\big(\widehat{A}^{\sharp}, \, \sigma^*(S^3\FF \otimes \bigwedge^2 \FF^{\vee}) \otimes \mathcal{O}_{\widehat{A}^{\sharp}}(-\Lambda)\big) \\
& \cong   H^0\big(\widehat{A}, \, \sigma_*(\sigma^*(S^3\FF \otimes \bigwedge^2 \FF^{\vee}) \otimes \mathcal{O}_{\widehat{A}^{\sharp}}(-\Lambda))\big) \\
& \cong   H^0\big(\widehat{A}, \, S^3\FF \otimes \bigwedge^2
\FF^{\vee} \otimes \mathcal{I}_{\hat{o}}\big),
\end{split}
\end{equation*}
hence $i=\tau \circ i^{\sharp}\colon S \hookrightarrow
\mathbb{P}(\FF)$ provides an embedding of $S$ as a divisor in the
linear system $|3\xi - \pi^*L_{\delta}|$ containing the fibre
$\pi^{-1}(\wo)$. Such a fibre must coincide with $\Xi$, because
$\Xi$ is the unique rational curve in $S$.
\end{proof}
Given the embedding $i \colon S \hookrightarrow \mathbb{P}$, the Albanese map
$\alpha \colon S \to \wA$ of $S$ factors as $\alpha= \pi \circ i$,
as in the following diagram:
\begin{equation} \label{dia.alpha}
\xymatrix{
S \ar@{^{(}->}[r]^i \ar[dr]^{\alpha} & \mathbb{P} \ar[d]^{\pi} \\
 & \wA. }
\end{equation}

Since $K_{\mP}=-2 \xi + \pi^*\Ld$, the adjunction formula implies
that the canonical line bundle of $S$ is the restriction of
$\oo_{\mathbb{P}}(1)$ to $S$, that is
\begin{equation} \label{eq.can.S}
\omega_S = \oo_S(\xi).
\end{equation}
In the sequel we shall exploit the following short exact sequences:\\
$\bullet$ the normal bundle sequence of  $i \colon S
\hookrightarrow \mP$, i.e.,
\begin{equation} \label{seq.normal}
0 \lr \oo_{\mP} \lr \oo_{\mP}(S) \lr N_{S/\mP} \lr 0;
\end{equation}
$\bullet$ the tangent bundle sequence of $i \colon S \hookrightarrow
\mP$, i.e.,
\begin{equation} \label{seq.tang.i}
0 \lr T_S \lr T_{\mP} \otimes \oo_S \lr N_{S/\mP} \lr 0;
\end{equation}
$\bullet$ the tangent bundle sequence of $\pi \colon \mP \to \wA$,
i. e.
\begin{equation} \label{seq.tang.pi}
0 \lr T_{\mP/ \wA} \lr T_{\mP}  \lr \pi^*T_{\wA} \lr 0.
\end{equation}
Recalling that $ S \in   |3\xi - \pi^*L_{\delta}|$, we get
\begin{equation*}
\pi_* \oo_{\mathbb{P}}= \oo_{\wA}, \quad R^1 \pi_* \oo_{\mP}=0, \quad
\pi_* \oo_{\mP}(S)= S^3 \FF \otimes \bigwedge^2 \FF^{\vee}, \quad
R^1 \pi_* \oo_{\mP}(S)=0,
\end{equation*}
so by the Leray spectral sequence we obtain
\begin{equation*}
H^i(\mP, \, \oo_{\mP})=H^i(\wA, \, \oo_{\wA}), \quad H^i(\mP, \,
\oo_{\mP}(S))=H^i(\wA, \, S^3 \FF \otimes \Lambda ^2 \FF^{\vee}),
\quad i \geq 0.
\end{equation*}
Hence, considering the long exact sequence associated with
\eqref{seq.normal} and using Proposition \ref{prop.coh.S3}, we
deduce
\begin{equation} \label{eq.coh.bundle}
h^0(S, \, N_{S/\mP})=3, \quad h^1(S, \, N_{S/\mP})=1, \quad h^2(S, \,
N_{S/ \mP})=0.
\end{equation}
Let us denote by $H^{\mP}_S$ the complex
analytic space representing the functor of
embedded deformations of $S$ inside
$\mP$ (with $\mP$ fixed), see \cite[p. 123]{Se06}.
An immediate consequence of
\eqref{eq.coh.bundle} is
\begin{prop} \label{prop.H.unobstructed}
$H^{\mP}_S$ is generically smooth, of dimension $3$.
\end{prop}
\begin{proof}
Since $h^0(S, \, N_{S/\mP})=3$, the dimension of the tangent space
of $H^{\mP}_S$ at the point corresponding to $S$ is $3$. On the
other hand, the family of embedded deformations of $S$ in $\mP$ is
at least $3$-dimensional: indeed, we can move $S$ into the
$1$-dimensional linear system $|\oo_{\mP}(S)|$ and we can
translate it by using the $2$-dimensional family of translations
of $\wA$. Therefore $H^{\mP}_S$ is smooth at $S$, hence
generically smooth of dimension $3$. In particular, the
obstructions in $H^1(S, \, N_{S/\mP})= \mathbb{C}$ actually
vanish.
\end{proof}
Now let us consider the long cohomology sequence associated with
\eqref{seq.tang.i}. Since $S$ is a surface of general type, we have
$H^0(S, \, T_S)=0$ and we get
\begin{equation} \label{seq.coh.tang.i}
\begin{split}
0 & \lr H^0(S, \, T_{\mP} \otimes \oo_S) \lr H^0(S, \, N_{S/\mP})
\stackrel{\delta^0}{\lr} H^1(S, \, T_S) \\
& \lr H^1(S, \, T_{\mP} \otimes \oo_S) \lr H^1(S, \, N_{S/\mP})
\stackrel{\delta^1}{\lr} H^2(S, \, T_S) \lr H^2(S, \, T_{\mP}
\otimes \oo_S) \lr 0.
\end{split}
\end{equation}
By standard deformation theory, see for instance \cite[p.
132]{Se06}, the map $\delta^0 \colon H^0(S, \, N_{S/\mP}) \to
H^1(S, \, T_S)$ is precisely the map induced on tangent spaces by
the ``forgetful morphism" $H^{\mP}_S \to \textrm{Def}(S)$, where
$\textrm{Def}(S)$ is the base of the Kuranishi family of $S$.

%Because of \eqref{seq.coh.tang.i}, we can interpret $H^1(S, \,
%T_{\mP} \otimes \oo_S)$ as the obstruction to lifting an abstract
%first-order deformation of $S$ to an embedded first-order
%deformation. We can also interpret the image of $H^0(S, \, T_{\mP}
%\otimes \oo_S)$ in $H^0(S, \, N_{S/\mP})$ as those first-order embedded
%deformations of $S$ induced by automorphisms of $\mP$.

Now we look at \eqref{seq.tang.pi}. Since $T_{\wA}$ is trivial, we
obtain
\begin{equation*}
T_{\mP / \wA} = \oo_{\mP}(-K_{\mP})=\oo_{\mP}(2 \xi - \pi^*\Ld).
\end{equation*}
Then $R^1 \pi_* T_{\mP / \wA} =0$, and Leray spectral sequence
together with \eqref{eq.S2=0} yields
\begin{equation}
H^i(\mathbb{P}, \, T_{\mP / \wA})=H^i(\wA, \, S^2 \FF \otimes
\bigwedge^2 \FF^{\vee})=0, \quad i \geq 0.
\end{equation}
Therefore $\textrm{Ext}^1(\pi^*T_{\wA}, \, T_{\mP /
\wA})=H^1(\mathbb{P}, \, T_{\mP / \wA})^{\oplus 2}=0$, so
 \eqref{seq.tang.pi} actually splits and we have
\begin{equation} \label{eq.pi.split}
T_{\mP}=T_{\mP/ \wA} \oplus \pi^*T_{\wA}= \oo_{\mP}(2 \xi -
\pi^*\Ld) \oplus \pi^*T_{\wA}.
\end{equation}
Since $N_{S/\mathbb{P}}=\oo_S(3 \xi - \pi^* L_{\delta})$, by
restricting \eqref{eq.pi.split} to $S$ and using \eqref{eq.can.S},
 we obtain
\begin{equation} \label{eq.pi.S.split}
T_{\mP} \otimes \oo_S=(T_{\mP/ \wA} \otimes \oo_S) \oplus
(\pi^*T_{\wA} \otimes \oo_S)= (N_{S/\mathbb{P}} \otimes \omega_S^{-1}) \oplus
\alpha^*\,T_{\wA}.
\end{equation}
Let us now compute the cohomology
groups of $N_{S/\mathbb{P}} \otimes \omega_S^{-1}=\oo_S(2 \xi - \pi^* \Ld)$. \\

\begin{lem} \label{prop.coh.N}
We have
\begin{equation*}
h^0(S, \, N_{S/\mathbb{P}} \otimes \omega_S^{-1})=0,
\quad h^1(S, \, N_{S/\mathbb{P}} \otimes \omega_S^{-1} )=0,
\quad h^2(S, \, N_{S/\mathbb{P}} \otimes \omega_S^{-1})=0.
\end{equation*}
\end{lem}
\begin{proof}
Let us consider the short exact sequence
\begin{equation*}
0 \lr \mathcal{O}_{\mathbb{P}}(-\xi) \lr \mathcal{O}_{\mathbb{P}}(2
\xi - \pi^*L_{\delta}) \lr \mathcal{O}_S(2 \xi - \pi^*L_{\delta})
\lr 0.
\end{equation*}
By \cite[p. 253]{H77} we have $\pi_* \mathcal{O}_{\mathbb{P}}
(-\xi)=R^1 \pi_* \mathcal{O}_{\mathbb{P}} (-\xi)=0$, so by Leray
spectral sequence we deduce $H^0(\mathbb{P},
\,\mathcal{O}_{\mathbb{P}}(-\xi))=H^1(\mathbb{P},
\,\mathcal{O}_{\mathbb{P}}(-\xi))=0$. It follows
\begin{equation*}
H^i(S, \, N_{S/\mathbb{P}} \otimes \omega_S^{-1})=H^i(\mathbb{P}, \,
\oo_{\mathbb{P}} (2 \xi - \pi^* \Ld))=H^i(\wA, \, S^2 \FF \otimes
\bigwedge^2 \FF^{\vee})=0
\end{equation*}
for $i=0, \, 1, \,2$, see \eqref{eq.S2=0}.
\end{proof}

By using \eqref{seq.coh.tang.i}, \eqref{eq.pi.S.split} and
Lemma \ref{prop.coh.N} we obtain the exact sequence
\begin{equation}
\begin{split}
0 & \lr H^0(S, \, \alpha^* T_{\wA}) \lr H^0(S, \, N_{S/\mP})
\stackrel{\delta^0}{\lr} H^1(S, \, T_S) \\
& \stackrel{\gamma}{\lr} H^1(S, \, \alpha^* T_{\wA} ) \lr H^1(S, \,
N_{S/\mP}) \stackrel{\delta^1}{\lr} H^2(S, \, T_S) \lr H^2(S, \,
\alpha^* T_{\wA}) \lr 0.
\end{split}
\end{equation}
The key remark is now contained in the following

\begin{prop} \label{prop.im.gamma}
The image of $\gamma \colon H^1(S, \, T_S) \to H^1(S, \, \alpha^*
T_{\wA})$ has dimension $3$.
\end{prop}
\begin{proof}
Since $T_{\wA}$ is trivial and there is a natural isomorphism
\begin{equation*}
H^1(S, \, \oo_S) \cong H^1(\wA, \, \oo_{\wA}),
\end{equation*}
we can see the map $\gamma$ as a map
\begin{equation*}
\gamma \colon H^1(S, \, T_S) \to H^1(\wA, \, T_{\wA}).
\end{equation*}
Take a positive integer $m \geq 2$ such that there exists a smooth
pluricanonical divisor $C \in |mK_S|$ and let $C'$ be
the image of $C$  in $\widehat{A}$; then we have a commutative
diagram
\begin{equation*} \label{dia.proof.2}
\begin{xy}
\xymatrix{%%
H^1(S, \, T_S(-\log \, C))  \ar[d]_{\epsilon} \ar[rr]^{\gamma'} & &
H^1(\wA, \, T_{\wA}(- \log \, C'))
 \ar[d]^{\epsilon'} \\
 H^1(S, \, T_S) \ar[rr]^{\gamma} & & H^1(\wA, \, T_{\wA}).  \\
  }
\end{xy}
\end{equation*}
Let us observe now the following facts.
\begin{itemize}
\item Since $S$ is smooth, the line bundle $\omega_S^m$ extends
along any first-order deformation of $S$, because the relative
dualizing sheaf is locally free for any smooth morphism of schemes, see
\cite[p. 182]{Man08}.
Moreover, since $S$ is minimal of general type, we have $h^1(S, \,
\omega_S^m)=0$, so every section of $\omega_S^m$ extends as well,
see \cite[Section 3.3.4]{Se06}. This means that no first-order
deformation of $S$ makes $C$ disappear, in other words $\epsilon$
is surjective. Therefore $im \,\gamma \subseteq im \, \epsilon'$.
\item Since $(C')^2 >0$, the line bundle $\oo_{\wA}(C')$ is ample
on $\wA$; therefore it deforms along a subspace of $H^1(\wA, \,
T_{\wA})$ of dimension $3$, see \cite[p. 152]{Se06}. Since every
first-order deformation of the pair $(\wA, \,C')$ induces a
first-order deformation of the pair $(\wA, \, \oo_{\wA}(C'))$, it
follows that the image of $\epsilon'$ is at most $3$-dimensional.
\end{itemize}
According to the above remarks, we obtain
\begin{equation*}
\dim \,(im \, \gamma) \leq \dim \,(im \, \epsilon') \leq 3.
\end{equation*}
On the other hand, given any abelian surface $\wA$ with a $(1,
\,2)$-polarization which is not of product type we may construct a
Chen-Hacon surface $S$ such that $\textrm{Alb}(S)=\wA$. Hence the
dimension of $im \, \gamma$ is exactly $3$.
\end{proof}

\begin{cor} \label{cor.dim.tang}
We have
\begin{equation*}
h^1(S, \, T_S)=4, \quad h^2(S, \, T_S)=4.
\end{equation*}
\end{cor}
\begin{proof}
By Riemann-Roch theorem we obtain $h^1(S, \, T_S)-h^2(S,
T_S)=10\chi(\oo_S)-2K_S^2=0$. On the other hand, Proposition
 \ref{prop.im.gamma} together with \eqref{eq.coh.bundle} implies
 $h^1(S, \, T_S)=4$, so we are done.
\end{proof}

Now let us denote by $\mathcal{M}$ the moduli space of minimal
surfaces of general type with $p_g=q=2$, $K_S^2=5$ and by
$\mathcal{M}^{CH}$ the subset of $\mathcal{M}$ given by isomorphism
classes of  Chen-Hacon surfaces.

\begin{theo} \label{teo.moduli}
$\mathcal{M}^{CH}$ is a connected, irreducible, generically smooth
 component of $\mathcal{M}$ of dimension $4$.
\end{theo}
\begin{proof}
The construction of Chen-Hacon surfaces depends on four parameters:
in fact, the moduli space $\mathcal{W}(1, \, 2)$ of $(1,
\,2)$-polarized abelian surfaces has dimension $3$, whereas
$\mathbb{P}H^0(\wA, \, S^3\FF \otimes \bigwedge^2 \FF^{\vee})$ is
$1$-dimensional (note that $\mathcal{F}$ does not give any
contribution to the number of parameters because of
\eqref{eq.ext.1}). This argument also shows that one has a
generically finite, dominant map
\begin{equation*}
\mathcal{W}(1, \, 2) \times \mathbb{P}H^0(\wA, \, S^3\FF \otimes
\bigwedge^2 \FF^{\vee}) \lr \mathcal{M}^{CH},
\end{equation*}
hence $\mathcal{M}^{CH}$ is an irreducible, algebraic subset of
$\mathcal{M}$ and $\dim \, \mathcal{M}^{CH}=4$. On the other hand,
if $K_S$ is ample Corollary \ref{cor.dim.tang} implies
\begin{equation*}
\dim \, T_{[S]} \mathcal{M}^{CH} = H^1(S, \, T_S)=4,
\end{equation*}
so $\mathcal{M}^{CH}$ is generically smooth.

It remains to show
that $\mathcal{M}^{CH}$ is a connected component of $\mathcal{M}$,
 i. e. that it is both open and closed therein.

 \bigskip

\emph{$\mathcal{M}^{CH}$ is open in $\mathcal{M}$.}

\medskip

Let
$\mathcal{S} \stackrel{\pi}{\to} \mathcal{B}$ be a deformation over
a small disk such that $S_0:=\pi^{-1}(0)$ is a Chen-Hacon surface.
We want to show that the same is true for $S_t:=\pi^{-1}(t)$. By Ehresmann's theorem,
$S_t$ is diffeomorphic to $S_0$, so it is a minimal surface of general type with $p_g=q=2$,
$K_{S_t}^2=5$. Moreover,
by \cite[p. 267]{Ca91}, the differentiable structure
of the general fibre of the Albanese map of $S_t$ is completely
determined by the differentiable structure of $S_t$;
it follows that the Albanese map $\alpha_t
\colon S_t \to \textrm{Alb}(S_t)$ is a generically finite triple
cover. Let
\begin{equation*}
S_t \stackrel{p_t}{\lr} X_t \stackrel{f_t}{\lr} \textrm{Alb}(S_t)
\end{equation*}
be the Stein factorization of $\alpha_t$, and let $\mathcal{E}_t$ be
the Tschirnhausen bundle associated with the flat triple cover $f_t
\colon X_t \to \textrm{Alb}(S_t)$, that is
\begin{equation} \label{eq.Ts.t}
f_{t*} \oo_{X_t}= \oo_{\textrm{Alb}(S_t)} \oplus \EE_t.
\end{equation}
By Proposition \ref{prop.quadruple}, $X_0$ has only rational
singularities, so the same holds for $X_t$ if $\mathcal{B}$ is
small enough.

The branch locus $\Delta_t$ of $\alpha_t$ is a deformation of
$\Delta_0$, in particular $p_a(\Delta_t)=p_a(\Delta_0)=9$; moreover,
by \cite[Proposition 4.7]{M85} the class of $\Delta_t$ must be
$2$-divisible in the Picard group of $\textrm{Alb}(S_t)$. It follows
that $\textrm{Alb}(S_t)$ is a $(1, \, 2)$-polarized abelian surface
and $\Delta_t \in |2L_t|$. Moreover $\bigwedge^2 \EE_t^{\vee}$ is
numerically equivalent to $\mL_t$, in particular $c_1^2(\EE_t)=4$.
Since $f_t$ is a finite map and $X_t$ has only rational
singularities, we obtain
\begin{equation*}
\begin{split}
h^1(\textrm{Alb}(S_t), \, f_{t*} \oo_{X_t}) & = h^1(X_t, \,
\oo_{X_t})=h^1(S_t, \, \oo_{S_t})=2, \\
h^2(\textrm{Alb}(S_t), \, f_{t*} \oo_{X_t}) & = h^2(X_t, \,
\oo_{X_t})=h^2(S_t, \, \oo_{S_t})=2,
\end{split}
\end{equation*}
so by using \eqref{eq.Ts.t} we deduce
\begin{equation*}
h^0(\textrm{Alb}(S_t), \, \EE_t)=0, \quad h^1(\textrm{Alb}(S_t),
\, \EE_t)=0, \quad h^2(\textrm{Alb}(S_t), \, \EE_t)=1.
\end{equation*}
Now Hirzebruch-Riemann-Roch Theorem yields
\begin{equation*}
1=\chi(\textrm{Alb}(S_t), \, \EE_t)= \frac{1}{2} c_1^2(\EE_t) -
c_2(\EE_t),
\end{equation*}
hence $c_2(\EE_t)=1$. It follows that the invariant of $S_t$ are
computed by formulae in Proposition \ref{prop.invariants}, in
other words $X_t$ contains only negligible singularities. By
Theorem \ref{teo.ch}, $S_t$ is a Chen-Hacon surface.

 \bigskip

\emph{$\mathcal{M}^{CH}$ is closed in $\mathcal{M}$.}

 \medskip

Let
$\mathcal{S} \stackrel{\pi}{\to} \mathcal{B}$ be a small deformation
 and assume that, for $t \neq 0$, $S_t$ is a Chen-Hacon
surface. We want to show that $S_0$ is a Chen-Hacon surface. Arguing
as before, we see that $\alpha_0 \colon S_0 \to \textrm{Alb}(S_0)$
is a generically finite triple cover, and that $\textrm{Alb}(S_0)$
is a $(1,\,2)$-polarized abelian surface. Let $\mathfrak{D}_t
\subset |2L_t|$ be the linear system
$\mathbb{P}H^0(\textrm{Alb}(S_t), \, \mathcal{L}_t^2 \otimes
\mathcal{I}_o^4)$. Since $\Delta_t \in \mathfrak{D}_t$ for all $t
\neq 0$, we have $\Delta_0 \in \mathfrak{D}_0$. The possible curves
in $|\mathfrak{D}_0|$ are classified in Proposition
\ref{prop.pencil.I4}; in all cases the corresponding triple cover
 contains only negligible singularities (see Examples \ref{ex.1}, \ref{ex.2}, \ref{ex.3}),
 so $S_0$ is a Chen-Hacon surface and we are done. \\ \\
This concludes the proof of Theorem \ref{teo.moduli}.
\end{proof}

Theorem \ref{teo.moduli} shows that every small deformation of a
Chen-Hacon surface is still a Chen-Hacon surface; in particular,
no small deformation of $S$ makes the $(-3)$-curve disappear.
Moreover, since $\mathcal{M}^{CH}$ is generically smooth, the same
is true for the first-order deformations. By contrast, Burns and
Wahl proved in \cite{BW74} that first-order deformations always
smooth all the $(-2)$-curves, and Catanese used this fact in
\cite{Ca89} in order to produce examples of surfaces of general
type with everywhere non-reduced moduli spaces. Theorem
\ref{teo.moduli} demonstrates rather strikingly that the results
of Burns-Wahl and Catanese cannot be extended to the case of
$(-3)$-curves and, as far as we know, it also provides the first
explicit example of this situation.
\medskip

%%%%%%%%%%%%%%%%%%%%%%%%%%%%%%%%%%%%%%%%%%%%%%%%%%%%%%%%%%%%%%%%%%%%%%%%%%%%%%%%%%%%%%%%%%%%

%%%%%%%%%%%%%%%%%%%%%%%%%%%%%%%%%%%%%%%%%%%%%%%%%%%%%%%%%%%5
%%%%%%%%%%%%%%%%%%%%%%%%%%%%%%%%%%%%%%%%%%%%%%%%%%%%%%%%%%%5
 %%%%%%%%%%%%%%%%%%%%%%%%%%%%%%%%%%%%%%%%%%%%%%%%%%%%%%%%%%%5
\bigskip
\bigskip

Matteo Penegini, Lehrstuhl Mathematik VIII, Universit\"at Bayreuth,
NWII, D-95440 Bayreuth, \\ Germany \\ \emph{E-mail address:}
 \verb|matteo.penegini@uni-bayreuth.de| \\ \\

Francesco Polizzi, Dipartimento di Matematica, Universit\`{a}  della
Calabria, Cubo 30B, 87036 \\ Arcavacata di Rende (Cosenza), Italy\\
\emph{E-mail address:} \verb|polizzi@mat.unical.it|

\end{document}